\RequirePackage{fix-cm}
\documentclass[smallextended]{svjour3_v2}       
\smartqed  
\usepackage{graphicx}
\usepackage{epstopdf}
\usepackage{amsmath, amsfonts, amssymb, algorithm, algpseudocode}
\usepackage{mathtools}
\usepackage{graphicx}
\usepackage{caption}
\usepackage{booktabs}
\usepackage{multirow}
\usepackage{enumerate}
\usepackage{framed}
\usepackage{tabularx}
\usepackage{color}
\usepackage{dsfont}
\usepackage{microtype}
\usepackage{setspace}
\usepackage{enumitem}
\setlist[enumerate]{itemsep = 0 mm}
\usepackage[normalem]{ulem}
\usepackage{placeins}

\DeclareMathOperator*{\argmin}{arg\min}
\DeclareMathOperator*{\argmax}{arg\max}

\def\grad{\nabla}

\def\ba{\mathbf{a}}

\def\bd{\mathbf{d}}

\def\bI{\mathbf{I}}

\def\cC{\mathcal{C}}

\def\cG{\mathcal{G}}

\def\cI{\mathcal{I}}

\def\cK{\mathcal{K}}
\def\cL{\mathcal{L}}

\def\cN{\mathcal{N}}
\def\cO{\mathcal{O}}
\def\cP{\mathcal{P}}
\def\cQ{\mathcal{Q}}

\def\cS{\mathcal{S}}

\def\cW{\mathcal{W}}

\def\smskip{\smallskip}

\def\texitem#1{\par\smskip\noindent\hangindent 25pt
               \hbox to 25pt {\hss #1 ~}\ignorespaces}

\def\norm#1{\left\|#1\right\|}

\newcommand{\BEAS}{\begin{eqnarray*}}
\newcommand{\EEAS}{\end{eqnarray*}}
\newcommand{\BEA}{\begin{eqnarray}}
\newcommand{\EEA}{\end{eqnarray}}
\newcommand{\BEQ}{\begin{eqnarray}}
\newcommand{\EEQ}{\end{eqnarray}}
\newcommand{\BIT}{\begin{itemize}}
\newcommand{\EIT}{\end{itemize}}
\newcommand{\BNUM}{\begin{enumerate}}
\newcommand{\ENUM}{\end{enumerate}}

\newcommand{\BA}{\begin{array}}
\newcommand{\EA}{\end{array}}

\newcommand{\ones}{\mathbf 1}

\newcommand{\reals}{\mathbb{R}}
\newcommand{\integers}{\mathbb{Z}}

\newcommand{\Rank}{\mathop{\bf rank}}

\newcommand{\diag}{\mathop{\bf diag}}

\newcommand{\dom}{\mathop{\bf dom}}

\newif\ifpagenumbering
\pagenumberingtrue

\pagenumberingfalse

\newsavebox{\theorembox}
\newsavebox{\lemmabox}
\newsavebox{\corollarybox}
\newsavebox{\remarkbox}
\newsavebox{\assbox}
\savebox{\theorembox}{\noindent\bf Theorem}
\savebox{\lemmabox}{\noindent\bf Lemma}
\savebox{\corollarybox}{\noindent\bf Corollary}
\savebox{\remarkbox}{\noindent\bf Remark}
\savebox{\assbox}{\noindent\bf Assumption}
\newtheorem{assumption}{\usebox{\assbox}}

\def\fprod#1{\left\langle#1\right\rangle}
\def\T{\mathsf{T}}
\def\id{\mathbf{I}}
\def\py{\pmb{y}}
\def\pxi{\pmb{\xi}}
\def\pt{\pmb{\theta}}
\def\pe{\pmb{\eta}}
\def\pss{\pmb{s}}
\def\one{\mathbf{1}}
\def\zer{\mathbf{0}}
\def\st{\mathrm{s.t.}}

\journalname{Computational Optimization and Applications}
\begin{document}

\title{A Parallelizable Dual Smoothing Method for Large Scale Convex Regression Problems\thanks{A preliminary version of this work~\cite{aybat2014parallel} has appeared in the Proceedings of the IEEE Conference on Decision and Control.}
}

\titlerunning{A Dual Smoothing Method for Convex Regression}        

\author{Necdet Serhat Aybat         \and
        Zi Wang 
}

\authorrunning{Aybat and Wang} 

\institute{N. S. Aybat \at
              Industrial and Manufacturing Engineering Dept., Penn State University\\
              University Park, PA 16802, USA
              \email{nsa10@psu.edu}           
           \and
           Z. Wang \at
           Industrial and Manufacturing Engineering Dept., Penn State University\\
           University Park, PA 16802, USA
           \email{zxw121@psu.edu}
}

\date{Date: 08/07/2016}

\maketitle

\begin{abstract}
\vspace*{-6mm}
Convex regression~(CR)
is an approach for fitting a convex function to a finite number of observations. It 
arises in various applications from diverse fields such as statistics, operations research, economics, and electrical engineering. The least squares~(LS) estimator, which can be computed via solving a quadratic program~(QP), is an intuitive method for convex regression with already established strong theoretical guarantees.
On the other hand, since the number of constraints in the QP formulation increases quadratically in the number of observed data points, the QP 
quickly becomes impractical to solve using traditional interior point methods. To address this issue, we propose a first-order method based on dual smoothing that carefully manages the memory usage through parallelization in order to efficiently compute the LS estimator in practice for large-scale CR instances.
\vspace*{-3mm}
\end{abstract}
\section{Introduction}
\emph{Convex regression}~(CR) problem 
deals with fitting a convex function to a given finite set of location/observation pairs, where each pair consists of a vector of independent variables and corresponding scalar dependent variable. In particular, suppose $N$ location/observation pairs are given $\{(\bar{x}_\ell,\bar{y}_\ell)\}_{\ell=1}^N\subset\mathbb{R}^n\times\mathbb{R}$ satisfying \vspace*{-4mm}
\begin{align}
\label{data}
\bar{y}_\ell=f_0(\bar{x}_\ell)+\varepsilon_\ell,\quad \ell=1,\ldots,N,
\end{align}
where $f_0: \mathbb{R}^n \rightarrow \mathbb{R}$ is a convex function, and $\varepsilon_\ell$ is a random noise with $E[\varepsilon_\ell]=0$ for all $\ell$. The objective is to infer the convex function $f_0$ from the noisy observations $\{(\bar{x}_\ell,\bar{y}_\ell)\}_{\ell=1}^N$. 
CR problems arise in various applications coming from diverse fields such as statistics, operations research, economics, and electrical engineering. 
M. Mousavi~\cite{mousavi2013shape} used 
CR to estimate the value function for Markov chains with expected infinite-horizon discounted rewards, which naturally arises in various control problems, and estimating value functions is essential for approximate dynamic programming and applied probability. In economics, CR has been adopted for approximating consumers' concave utility functions from empirical data~\cite{meyer1968consistent}. 
Moreover, in queueing network context, when the expectation of a performance measure is convex in model parameters -- see
~\cite{chen2001fundamentals}, then using Monte Carlo methods to compute the expectation gives rise to a CR problem~\cite{lim2012consistency}.

CR was first studied in~\cite{hildreth1954point} for estimating concave production functions. Later, various solution methods were proposed in the uni-variate setting, e.g.,~\cite{dent1973note,birke2007estimating,shively2011nonparametric}. The problem of fitting a convex function in the multi-variate setting has been considered in~\cite{holloway1979technical,kuosmanen2008representation} where the minimization of the least squares~(LS) error subject to the first-order convexity shape constraints is studied; furthermore, \cite{aguilera2008approximating,aguilera2009convex} also considered the same approach with additional second-order convexity constraints. The most well-known method for CR is to solve the LS problem,\vspace*{-5mm}
\begin{equation}
\label{infinite_dim}
\hat{f}_N=\argmin_{f\in\cC} \sum\limits_{\ell=1}^{N} \big( f(\bar{x}_\ell ) - \bar{y}_\ell \big)^2, \vspace*{-1mm}
\end{equation}
where $\cC\triangleq\{f:\mathbb{R}^n\rightarrow\mathbb{R} \hbox{ such that } f \hbox{ is convex}\}$. This \emph{infinite} dimensional problem 
is equivalent to a finite dimensional quadratic problem~(QP) given in~\eqref{original} -- see Proposition~1 in~\cite{lim2012consistency}, \vspace*{-2mm}
{\small
\begin{equation}
\label{original}
\min_{\substack{y_\ell\in\reals,~\xi_\ell\in\reals^n\\ \ell=1,\ldots,N}}\left\{\sum\limits_{\ell=1}^{N}   \big| y_\ell -\bar{y}_\ell \big|^2:\ {y}_{\ell_2} - y_{\ell_1} + {\xi_{\ell_1}}^\top (\bar{x}_{\ell_1}-\bar{x}_{\ell_2})\geq 0, \quad  1\leq \ell_1 \neq \ell_2\leq N\right\}.
\end{equation}}%
Indeed, let $\{(y_\ell^*,\xi_\ell^*)\}_{\ell=1}^N$ be an optimal solution to \eqref{original}, it is easy to show that when $N \geq n + 1$, $\{y_\ell^*\}_{\ell=1}^N$ is unique, $\hat{f}_N(\bar{x}_\ell)=y_\ell^*$ and $\xi_\ell^*\in\partial\hat{f}_N(\bar{x}_\ell)$ for all $\ell$, where $\partial$ denotes the subdifferential operator. The theoretical behavior of the LS estimator has been studied thoroughly in the past 50 years. In the \emph{univariate} setting, i.e., $n=1$, the consistency of the LS estimator is proved in~\cite{hanson1976consistency}; and the convergence rate of the estimator is established in~\cite{mammen1991nonparametric}. Groeneboom et al.~\cite{groeneboom2001estimation} extended these results and derived the asymptotic distribution of LS estimator at a fixed point of positive curvature. In the \emph{multivariate} setting, the consistency 
is shown in~\cite{lim2012consistency}, i.e., $\hat{f}_N\rightarrow f_0$ almost surely as $N$ increases.

Besides LS estimator, there are other methods for solving 
CR problem in the multivariate setting. A heuristic approach is proposed in~\cite{magnani2009convex} to compute locally optimal fits, which has no convergence guarantee. A convex adaptive partitioning~(CAP) method is proposed in~\cite{hannah2011approximate}, which creates a globally convex regression model via computing locally linear fits on adaptively selected covariate partitions. Both methods use the piecewise linear model, and minimize the least square error. In addition, more recently, Hannah and Dunson~\cite{hannah2012ensemble} proposed a new estimator based on using traditional ensemble methods to average over multiple piecewise linear estimators, and proved its consistency when CAP is the underlying estimator. However, LS estimator has some significant advantages over the methods mentioned above. First, LS estimator is a non-parametric regression method as discussed in \cite{Sen11}, which does not require any tuning parameters and avoids the issue of selecting an appropriate estimation structure; however, as also pointed out in~\cite{mousavi2013shape}, the methods proposed in~\cite{hannah2011approximate,hannah2013multivariate} are semi-parametric, and require adjusting several parameters before fitting a convex function. Second, LS estimator can be computed by solving the QP in \eqref{original}; therefore, at least in theory, it can be solved very efficiently using interior point methods (IPM). A major drawback of the LS estimator in practice is that the number of shape constraints in \eqref{original} is $\mathcal{O}(N^2)$. Consequently, the problem quickly becomes massive even for moderate number of observations: for off-the-shelf IPMs that do not exploit any structural properties of \eqref{original}, the complexity of each factorization step is $\mathcal{O}(N^3(n+1)^3)$, 
and the memory requirement is $\cO \big( N^2(n+1)^2 \big)$ assuming Cholesky factors are stored - see \cite{Boyd04,Nocedal2006} -- for more detailed discussion on memory usage and computational complexity of both IPM and our proposed method (exploiting the structure), see~Section~\ref{sec:complexity}.\vspace*{-1mm}

\vspace*{-2mm} In this paper, we propose a new \emph{parallelizable} method for computing the LS estimator on large-scale CR problems. The proposed method can efficiently solve large-scale instances of \eqref{original} by carefully managing the memory usage through parallelization, and exploiting the underlying problem structure. 
In particular, the proposed method, P-APG, is based on dual smoothing, i.e., regularizing the objective in~\eqref{original} with a strongly convex function. More specifically, we adopted Tikhonov regularization, which leads to a differentiable dual function with a Lipchitz continuous gradient. Compared to the traditional dual decomposition methods, the dual smoothing based approaches can guarantee feasibility of primal iterate sequence in the limit. To briefly summarize, P-APG is an iterative method to solve the regularized QP problem in \eqref{regularize} through solving a number of small-size QPs in each iteration.
In 
our main results, Theorem~\ref{bound} and \ref{thm:xi-bound}, we establish error bounds on the quality of inexact solutions to the regularized problem; particularly, we investigate how well the inexact solutions can approximate i) function values of the LS estimator, i.e., $\hat{f}_N(\bar{x}_\ell)$, and ii) subgradients from the subdifferential of the LS estimator, i.e., $\partial \hat{f}_N(\bar{x}_\ell)$. Next, we study the convergence behavior of P-APG to compute these function value and subgradient approximations. In Section~\ref{sec:cont-method}, we show that using a \emph{continuation} method, we can construct an iterate sequence that is asymptotically optimal to the original LS problem in \eqref{original} with a provable convergence rate. We adopted a primal-dual IPM to solve the small-size QP subproblems arising in P-APG iterations, and analyzed the computational complexity of an P-APG iteration by exploiting the special structure of the constraints and the objective function. In the rest, 
as alternatives to P-APG, we examined how an active set method~(ASM) can be efficiently implemented to solve \eqref{original}, and briefly discussed a recently proposed ADMM algorithm~\cite{mazumder2015computational} for 
\eqref{original}.
Finally, we conclude with a number of numerical examples comparing P-APG, ASM, and ADMM. Our results show that P-APG is the method of choice for large $N$. 
\paragraph{Notations:} 
Throughout, i.i.d. is short for independent and identically distributed. 
$\id_n$ denotes the $n\times n$-identity matrix. Given $x\in\reals^n$, $(x)_{+}\triangleq\max\{ x,\ 0\}$ and $(x)_{-}\triangleq\min\{ x,\ 0\}$; hence, $x=(x)_++(x)_-$. For $x,y\in\reals^n$, $\fprod{x,y}\triangleq x^\T y$ represents the standard inner product. $\mathbf{1}$ denotes the vector of all ones, and $e_i\in\reals^n$ denotes the $i$-th unit vector for each $i\in\{1,\ldots,n\}$.
\section{Methodology}
Let $f_0:\mathbb{R}^n\rightarrow\mathbb{R}\cup\{+\infty\}$ be the \emph{unknown} proper convex function generating the observed data $\{(\bar{x}_\ell,\bar{y}_\ell)\}_{\ell=1}^N\subset\mathbb{R}^n\times\mathbb{R}$ as in \eqref{data}, and 
let $\cN:=\{1,\ldots,N\}$ denote the set of indices corresponding to $N$ observations. Suppose $B_x>0$ such that $\norm{\bar{x}_\ell}_2\leq B_x$ for all $\ell\in\cN$. Define the long-vector notations for the variables: $\py=[y_\ell]_{\ell\in\cN}\in\reals^N$, and $\pxi=[\xi_\ell]_{\ell\in\cN}\in\reals^{Nn}$.

Consider \eqref{original} in the following compact form:
\begin{align}
\label{original_compact}
\chi^*\triangleq\argmin_{\pmb{y}\in\reals^N,~\pmb{\xi}\in\reals^{Nn}}\left\{\tfrac{1}{2} \left\| \pmb{y} - \bar{\pmb{y}} \right\| _2^2:\ A_1~\pmb{y} +  A_2~\pmb{\xi}\geq 0\right\},
\end{align}
where $A_1\in\mathbb{R}^{N(N-1)\times N}$ and $A_2\in\mathbb{R}^{N(N-1)\times Nn}$ are the matrices corresponding to constraints in \eqref{original}. Let $(\pmb{y}^*, \pmb{\xi}^* )$ be the least-norm optimal solution in $\chi^*$, i.e.,
\begin{align}
\label{tikhonov}
(\pmb{y}^*, \pmb{\xi}^* ) \triangleq \argmin\limits_{\pmb{y},~\pmb{\xi}} \left\{   \tfrac{1}{2} \big\| \pmb{y} \big\| _2^2 + \tfrac{1}{2} \big\| \pmb{\xi} \big\| _2^2:\  (\pmb{y}, \pmb{\xi} ) \in \chi^*  \right\}.
\end{align}
It is easy to show that $\pmb{y}^*$ is unique to~\eqref{original_compact}, i.e., if $(\pmb{y}, \pmb{\xi})\in\chi^*$, then $\pmb{y}=\pmb{y}^*$ -- see Proposition~1 in~\cite{lim2012consistency}. Hence, it follows from \eqref{tikhonov} that $\pmb{\xi}^*$ has the least norm
, i.e., for all $(\pmb{y}, \pmb{\xi})\in\chi^*$, one has $\norm{\pmb{\xi}}_2\geq\norm{\pmb{\xi}^*}_2$. 
Moreover, since \eqref{original_compact} is a convex QP, strong duality holds, and an optimal dual solution $\pmb{\theta}^*\in\mathbb{R}^{N(N-1)}$ exists. 

Note for each $(\ell_1,\ell_2)\in\cP\triangleq\{(\ell_1,\ell_2)\in\cN\times\cN:\ \ell_1\neq\ell_2\}$, there is a constraint in \eqref{original}, i.e., ${y}_{\ell_2} - y_{\ell_1} + {\xi_{\ell_1}}^\top (\bar{x}_{\ell_1}-\bar{x}_{\ell_2})\geq 0$ corresponds to $(\ell_1,\ell_2)\in\cP$. In order to fix 
$A_1$ and $A_2$, we sort the rows according to increasing lexicographic order on the index set $\cP$, i.e., the row for the constraint corresponding to $(\ell_1,\ell_2)$ comes before than the one corresponding to $(\ell_3,\ell_4)$ if either $\ell_1<\ell_3$, or $\ell_2<\ell_4$ in case $\ell_1=\ell_3$. Next, we give explicit forms for $A_1$ and $A_2$.
\begin{definition}
\label{def:A1A2}
Let $T_\ell\in\reals^{N-1\times N}$ such that $T_\ell=[e_1 \cdots e_{\ell-1}~-\one~e_\ell \cdots e_{N-1}]$ for $\ell\in\cN$, where $e_j\in\reals^{N-1}$ is the $j$-th unit vector for $j\in\{1,\ldots,N-1\}$. Moreover, let $\bar{X}\in\reals^{N\times n}$ such that $\bar{X}=[\bar{x}_\ell^\top]_{\ell\in\cN}$, i.e., $\{\bar{x}_\ell\}_{\ell\in\cN}$ are the rows of $\bar{X}$. Then $A_1=[T_\ell]_{\ell\in\cN}$, obtained by vertically concatenating $\{T_\ell\}_{\ell\in\cN}$, and $A_2=\diag\left(\{-T_\ell\bar{X}\}_{\ell\in\cN}\right)$ is a block-diagonal matrix as given below
\begin{equation}
\label{eq:A}
A_1=
\left(
  \begin{array}{c}
  T_1\\
  T_2\\
  \vdots\\
  T_N\\
  \end{array}
\right),\quad
A_2=
\left(
  \begin{array}{cccc}
    X_1 & \mathbf{0} & \cdots & \mathbf{0} \\
    \mathbf{0} & X_2 & \cdots & \mathbf{0} \\
    \vdots & \vdots & \ddots & \vdots \\
    \mathbf{0} & \mathbf{0} & \cdots & X_N \\
  \end{array}
\right),\quad
X_\ell\triangleq-T_\ell\bar{X},\quad \ell\in\cN. \vspace*{-4mm}
\end{equation}
\end{definition}
\vspace*{-2mm}
\subsection{Separability}
Given the regularization parameter $\gamma\geq 0$, consider
\begin{align}
\label{regularize}
(\py^*_\gamma,\pxi^*_\gamma)\triangleq\argmin_{\pmb{y},~\pmb{\xi}} \left\{r_\gamma(\pmb{y},\pmb{\xi})\triangleq\frac{1}{2} \left\| \pmb{y} - \bar{\pmb{y}} \right\| _2^2 + \frac{\gamma}{2} \left\| \pmb{\xi} \right\| _2^2:\ A_1\,\pmb{y} +  A_2 \, {\pmb{\xi}}\geq 0\right\}.
\end{align}
Note simply setting $\gamma=0$ in \eqref{regularize}, we obtain the original problem~\eqref{original_compact}.

To reduce the \textit{curse of dimensionality} and develop a 
\emph{parallelizable} method that can solve problems in~\eqref{original_compact} and \eqref{regularize} for large $N$, we employ dual decomposition to induce \emph{separability}. To this aim, we \emph{partition} the observation set into $K$ subsets $\{\cC_i\}_{i\in\cK}$, where $\cK\triangleq\{1,\ldots,K\}$ denote the set of indices corresponding to $K$ subsets of $\cN$. In particular, we choose $\{\cC_i\}_{i\in\cK}$ as a partition of $\cN$ such that $|\cC_i|\geq n+1$ for all $i$. To simplify the notation, throughout the paper we make the following assumption.
\begin{assumption}
\label{assump:partition}
Suppose $N=K \bar{N}$ for some $\bar{N}>n+1$, and without loss of generality assume that $\cC_i\triangleq \big \{(i-1)\bar{N}+1,~(i-1)\bar{N}+2,\ldots,~i\bar{N} \big\}$ for $i\in\cK$.
\end{assumption}
Throughout the paper, for each $i\in\cK$, let $\pmb{y}_i\in\mathbb{R}^{\bar{N}}$ and $\pmb{\xi}_i\in\mathbb{R}^{\bar{N}n}$ denote the sub-vectors of $\pmb{y}\in\mathbb{R}^{N}$ and $\pmb{\xi}\in\mathbb{R}^{Nn}$ corresponding to indices in $\cC_i$, respectively. In particular, for all $i\in\cK$, $\pmb{y}_i=[y_\ell]_{\ell\in\cC_i}$ and $\pmb{\xi}_i=[\xi_\ell]_{\ell\in\cC_i}$. Similarly, we define the same long-vectors for the observation data: $\bar{\pmb{y}}_i=[\bar{y}_\ell]_{\ell\in\cC_i}\in\reals^{\bar{N}}$.
\begin{definition}
\label{eq:submatrices}
Define $\cP\triangleq\{(\ell_1,\ell_2)\in\cN\times\cN:\ \ell_1\neq\ell_2\}$ and $\cG\triangleq\{(i,j)\in\cK\times\cK:\ i \neq j\}$. For each $i\in\cK$, let $A_1^{ii}\in\reals^{\bar{N}(\bar{N}-1)\times N}$ and $A_2^{ii}\in\reals^{\bar{N}(\bar{N}-1)\times N n}$ be the submatrices of $A_1$ and $A_2$ such that they consist of the rows corresponding to row indices $(\ell_1,\ell_2)\in\cP$ for 
$\ell_1,\ell_2\in\mathcal{C}_i$. Similarly, for each 
$(i,j)\in\cG$, let ${A_1}^{ij}\in\reals^{\bar{N}^2\times N}$ and ${A_2}^{ij}\in\reals^{\bar{N}^2\times N n}$ be the submatrices of $A_1$ and $A_2$ consisting of the rows corresponding to indices $\{(\ell_1,\ell_2)\in\cP:\ \ell_1\in\cC_i,~\ell_2\in\cC_j\}$.

Furthermore, for each $i\in\cK$, let $\bar{A}_1^{ii}\in\reals^{\bar{N}(\bar{N}-1)\times \bar{N}}$ and $\bar{A}_2^{ii}\in\reals^{\bar{N}(\bar{N}-1)\times \bar{N}n}$ be the submatrices of $A_1^{ii}$ and $A_2^{ii}$ such that
 $\bar{A}_1^{ii}$ consists of the columns of $A_1^{ii}$ corresponding to $\py_i$; and $\bar{A}_2^{ii}$ consists of the columns of $A_2^{ii}$ corresponding to $\pxi_i$.
\end{definition}
Note that for every ordered pair $(\ell_1,\ell_2)\in\cP$, there corresponds a constraint in \eqref{original}, which is represented by a row in matrices $A_1$ and $A_2$ of formulations \eqref{original_compact} and \eqref{regularize}. Consider all the constraints in \eqref{original} corresponding to those pairs $(\ell_1,\ell_2)$ such that they belong to different sets in the partition, i.e., $\ell_1\in\cC_{i}$, $\ell_2\in\cC_{j}$ for some $(i,j)\in\cG$, 
let $\pmb{\theta}_{ij}\in\mathbb{R}^{\bar{N}^2}$ denote the associated dual variables, 
and $\pmb{\theta}=[\pt_{ij}]_{(i,j)\in\cG}\in\reals^{\bar{N}^2K(K-1)}$ denote the vector formed by vertically concatenating $\pmb{\theta}_{ij}$ for $1\leq i\neq j\leq K$. By dualizing all such constraints in \eqref{regularize}, we form the \emph{partial} Lagrangian function:
\begin{align}
\label{eq:lagrangian}
\mathcal{L}_{\gamma} \left(\pmb{y}, \pmb{\xi}, \pmb{\theta}\right) \triangleq &
\frac{1}{2}\sum\limits_{i\in\cK}\left( \big\| \pmb{y}_i - \bar{\pmb{y}}_i\big\|_2^2+\gamma\left\| \pmb{\xi}_i \right\| _2^2\right) - \sum_{(i,j)\in\cG}
\fprod{\pmb{\theta}_{ij},  A_1^{ij} \py +A_2^{ij} \pxi }.
\end{align}
and obtain the following partial dual function
\begin{align}
\label{subgradient}
g_\gamma(\pmb{\theta})\triangleq \min\limits_{\: \pmb{y},~\pmb{\xi}}\left\{\mathcal{L}_\gamma \left(\pmb{y}, \pmb{\xi}, \pmb{\theta}\right):\  A_1^{ii}\py+A_2^{ii}\pxi\geq 0,\ i\in\cK\right\}.
\end{align}
Hence, the dual problem corresponding to \eqref{regularize} is given as
\begin{align}
\label{eq:dual-problem}
\Theta^*_\gamma\triangleq\argmax\{g_\gamma(\pmb{\theta}):\ \pmb{\theta}\geq 0\},\quad \hbox{and}\quad p^*_\gamma\triangleq g_\gamma(\pmb{\theta}^*_\gamma)\quad \hbox{for}\quad \pmb{\theta}^*_\gamma\in\Theta^*_\gamma.
\end{align}
Since strong-duality trivially holds between the primal-dual problem pair, \eqref{regularize} and \eqref{eq:dual-problem}, we have
\begin{align}
\label{eq:strong-duality}
p^*_\gamma=r_\gamma(\py^*_\gamma,\pxi^*_\gamma)=\tfrac{1}{2} \left\| \py^*_\gamma - \bar{\pmb{y}} \right\| _2^2 + \tfrac{\gamma}{2} \left\| \pxi^*_\gamma \right\| _2^2.
\end{align}
For any given regularization parameter $\gamma\geq 0$ and dual variable $\pt$, the partial Lagrangian function $\cL_\gamma$ is \emph{separable} in $\{(\py_i,\pxi_i)\}_{i\in\cK}$, and can be written as
\begin{align}
\label{eq:separable-form}
\cL_\gamma \left(\pmb{y}, \pmb{\xi}, \pmb{\theta}\right)= \sum_{i\in\cK} \cL^i_\gamma\left(\pmb{y}_i, \pmb{\xi}_i, \pmb{\theta}\right)
\end{align}
for some very simple quadratic function, $\cL^i_\gamma$, of $(\py_i,\pxi_i)$ for each $i\in\cK$. Moreover, after partially dualizing some of the constraints as shown in \eqref{eq:lagrangian}, the remaining ones in \eqref{subgradient} define a superset, $\cQ$, of the original feasible region. Indeed, $\cQ=\{(\py,\pxi): A_1^{ii}\py+A_2^{ii}\pxi\geq 0,\ i\in\cK\}=\{(\py,\pxi):\ \bar{A}_1^{ii}\py_i+\bar{A}_2^{ii}\pxi_i\geq 0,\ i\in\cK\}$ -- since the entries of $A_1^{ii}$ that does not belong to its submatrix $\bar{A}_1^{ii}$ are all 0; and similarly, the entries of $A_2^{ii}$ that does not belong to its submatrix $\bar{A}_2^{ii}$ are all 0 as well. Therefore, we have $\cQ=\bigotimes_{i\in\cK}\cQ_i$, where $\cQ_i\triangleq\{(\py_i,\pxi_i):\ \bar{A}_1^{ii}\py_i+\bar{A}_2^{ii}\pxi_i\geq 0\}$ for $i\in\cK$, and $\bigotimes$ denotes the Cartesian product. Consequently, since $\cL_\gamma$ is separable as shown in \eqref{eq:separable-form}, computing the partial dual function $g_\gamma(\pmb{\theta})$ in \eqref{subgradient} is equivalent to solving $K$ quadratic subproblems, i.e., one for each $i\in\cK$,
\begin{align}
\label{subgradient_split}
\min_{\pmb{y}_i\in\reals^{\bar{N}},~\pmb{\xi}_i\in\reals^{\bar{N}n}}\left\{\mathcal{L}^i_\gamma (\pmb{y}_i, \pmb{\xi}_i, \pmb{\theta}):\  \bar{A}_1^{ii}~\pmb{y}_i + \bar{A}_2^{ii}~\pmb{\xi}_i \geq 0\right\}.
\end{align}
Given the dual variables $\pmb{\theta}$, since all $K$ subproblems can be computed in \emph{parallel}, 
one can take advantage of the computing power of multi-core processors. 
In the rest of the paper, we discuss 
how to compute a solution to \eqref{original} via solving the dual problem: $\max\{g_\gamma(\pmb{\theta}):\ \pmb{\theta}\geq 0\}$.

\subsection{Projected Subgradient Method for Dual}
\label{sec:subgradient}
Clearly, for $\gamma=0$, $g_0$ defined in \eqref{subgradient} is the dual function for the original problem~\eqref{original_compact}; and the projected subgradient method can be adopted for solving the dual problem $\max\{g_0(\pmb{\theta}):\ \pmb{\theta}\geq 0\}$. 
Let $\pt=\mathbf{0}$, i.e., $\pmb{\theta}^{0}_{ij}=\pmb{0}$ for all 
$(i,j)\in\cG$. Given the $k$-th dual iterate $\pmb{\theta}^k$, let $( \pmb{y}^k, \pmb{\xi}^k )$ denote \emph{an} optimal solution to the minimization problem in \eqref{subgradient} when $\gamma=0$ and $\pmb{\theta}$ is set to $\pmb{\theta}^k$; 
and let $\pmb{\theta}_{ii}^k$ denote an optimal dual associated with constraints $A_1^{ii}~\pmb{y} + A_2^{ii}~\pmb{\xi} \geq 0$ in \eqref{subgradient}. The next dual iterate $\pmb{\theta}^{k+1}$ is computed for an appropriately chosen step size $t_k>0$:
\begin{align}
\pmb{\theta}_{ij}^{k+1}  =  {\textstyle \prod_{\cS_{ij}^k}} \left( \pmb{\theta}_{ij}^{k} - t_k\left(  A_1^{ij} \py^k +A_2^{ij} \pxi^k \right) \right),
\end{align}
where $\Pi_{\cS_{ij}^k}(.)$ denotes the Euclidean projection on to
$$\cS_{ij}^k \triangleq \Big\{ \pmb{\theta}_{ij} \geq \pmb{0}: {\pmb{\theta}_{ij}}^{\mathsf{T}} A_2^{ij} +  {\pmb{\theta}_{ii}^k}^{\mathsf{T}} A_2^{ii}  =\pmb{0} \Big\}.$$
Since the Lagrangian function $\cL_0$ is linear in $\pmb{\xi}$ when $\gamma=0$, $\dom g_0$ is non-trivial; hence the projection on to the Cartesian product $\bigotimes_{(i,j)\in\cG}\cS_{ij}^k$ ensures $\pt^{k+1}\in \dom g_0$.
The projected subgradient method is guaranteed to converge in function value for a diminishing step size sequence $\{ t_k \}_{k=1}^{\infty}$, and it requires $\mathcal{O}( 1/ \epsilon^2)$ iterations to obtain an $\epsilon$-optimal solution --see~\cite{Nesterov04_1B}. 
On the other hand, even if the dual iterates converge to an optimal dual solution $\pmb{\theta}^*$, the primal feasibility of the corresponding primal iterate sequence $\{( \pmb{y}^k, \pmb{\xi}^k )\}$ cannot be guaranteed in the limit as it might converge to a stationary point of the Lagrangian $\cL_0(\cdot,\cdot,\pmb{\theta}^*)$ that is primal infeasible, mainly due to lack of strict convexity, jointly in $(\pmb{y},\pmb{\xi})$, of the objective 
in \eqref{original_compact}.
\subsection{Tikhonov Regularization Approach}
In order to ensure feasibility in the limit, which cannot be guaranteed by the subgradient method discussed above, we employ Tikhonov regularization as in \eqref{regularize} for $\gamma>0$, of which convergence properties in general were investigated in \cite{engl1989convergence}. In particular, as $\gamma$ decreases to zero from above, the minimizer $(\py^*_\gamma,\pxi^*_\gamma)$, as a function of $\gamma$, converges to $(\pmb{y}^*, \pmb{\xi}^* )\in\chi^*$ defined in \eqref{tikhonov}, i.e., $\pmb{\xi}^*$ has the least norm among all $(\pmb{y}^*,\pmb{\xi})\in\chi^*$.
\begin{lemma}
\label{lem:continuity}
The minimizer of \eqref{regularize}, $\py^*_\gamma$, as a function of the regularization parameter $\gamma$, is H\"{o}lder  continuous from right at $\gamma=0$. In particular,
\begin{equation}
\label{holder_cond}
\norm{\py^*_\gamma - \pmb{y}^*}_2 \leq \norm{\pmb{\xi}^*}_2 \sqrt{\gamma},\quad \forall \gamma\geq 0.
\end{equation}
\end{lemma}
\proof
Let $\left(\py^*_\gamma,\pxi^*_\gamma\right)$ be the optimal solution to \eqref{regularize} and $\left(\pmb{y}^*,\pmb{\xi}^*\right)$ be defined as in \eqref{tikhonov}. Note that $(\pmb{y}^*,\pmb{\xi}^*)$ and $(\py^*_\gamma,\pxi^*_\gamma)$ are feasible to \eqref{regularize} and \eqref{original_compact}, respectively; hence, from the first-order optimality conditions of \eqref{regularize} and \eqref{original_compact}, we have
\begin{equation}
\label{optimality condition}
\begin{pmatrix} \py^*_\gamma - \bar{\pmb{y}} \\ \gamma~\pxi^*_\gamma \end{pmatrix} ^{\mathsf{T}} \begin{pmatrix} \pmb{y}^* - \py^*_\gamma\\ \pmb{\xi}^* - \pxi^*_\gamma \end{pmatrix} \geq 0, \quad
\begin{pmatrix} \pmb{y}^* - \bar{\pmb{y}} \\ \mathbf{0} \end{pmatrix} ^{\mathsf{T}} \begin{pmatrix} \py^*_\gamma - \pmb{y}^*  \\ \pxi^*_\gamma - \pmb{\xi}^* \end{pmatrix} \geq 0.
\end{equation}
Moreover, since $\left(\pmb{y}^*,\pmb{\xi}^*\right)$ and $\left(\py^*_\gamma,\pxi^*_\gamma\right)$ are optimal to \eqref{original_compact} and \eqref{regularize}, respectively; we also have
\begin{align*}
\tfrac{1}{2}\norm{\pmb{y}^* - \bar{\pmb{y}}}_2^2\leq\tfrac{1}{2}\norm{\py^*_\gamma - \bar{\pmb{y}}}_2^2,
\quad
\tfrac{1}{2}\norm{\py^*_\gamma-\bar{\pmb{y}}}_2^2+\tfrac{\gamma}{2}\norm{\pxi^*_\gamma}_2^2
\leq\tfrac{1}{2}\norm{\pmb{y}^* - \bar{\pmb{y}}}_2^2+\tfrac{\gamma}{2}\norm{\pmb{\xi}^*}_2^2.
\end{align*}
These two inequalities imply $\norm{\pxi^*_\gamma}_2\leq \norm{\pmb{\xi}^*}_2$. Finally, summing the two inequlities in \eqref{optimality condition} and using Cauchy-Schwarz, we obtain
\begin{align*}
\left\|  \py^*_\gamma -  \pmb{y}^* \right\|_2^2 \leq \gamma \, {\pxi^*_\gamma}^{\mathsf{T}} \big(  \pmb{\xi}^* - \pxi^*_\gamma \big)\leq\gamma \left(\norm{\pmb{\xi}^*}_2^2-\norm{\pxi^*_\gamma}_2^2\right)\leq\gamma \norm{\pmb{\xi}^*}_2^2,
\end{align*}
which implies the desired result.\qed
\endproof

\vspace{-0.2 cm}
Since the objective function in \eqref{regularize} is strongly convex, jointly in $\pmb{y}$ and $\pmb{\xi}$, when $\gamma>0$, Danskin's theorem (see~\cite{Bertsekas99}) implies that $g_\gamma$, i.e., the Lagrangian dual function corresponding to \eqref{regularize}, is differentiable; therefore, one can use gradient type methods to solve the corresponding dual problem $\max\{g_\gamma(\pmb{\theta}):\ \pmb{\theta}\geq 0\}$. Moreover, strong convexity ensures that, one can solve the regularized primal problem in \eqref{regularize} by solving the associated dual problem in \eqref{eq:dual-problem}. Indeed, let $\pt^*_\gamma$ be an optimal solution to \eqref{eq:dual-problem}, we can recover $(\py^*_\gamma,\pxi^*_\gamma)$ by computing the primal minimizers in \eqref{subgradient} when the dual is set to $\pt^*_\gamma$. In particular, achieving primal feasibility in the limit for the primal iterate sequence is not an issue provided that we can construct a dual iterate sequence that is asymptotically optimal to \eqref{eq:dual-problem}. We complete this section by formally stating this result. 
\begin{theorem}
\label{thm:rate}
Let $\gamma>0$, and $\{\pmb{\theta}^k\}$ be some dual sequence such that $\pmb{\theta}^k\geq 0$ for $k\geq 1$ and $\lim_{k\in\integers_+}g_\gamma(\pmb{\theta}^k)=p^*_\gamma$. Moreover, let $( \pmb{y}^k, \pmb{\xi}^k )$ denote the \emph{unique} optimal solution to the minimization problem in \eqref{subgradient} when $\pmb{\theta}$ is set to $\pmb{\theta}^k$ for $k\geq 1$. Then $(\py^*_\gamma,\pxi^*_\gamma)$ is the unique limit point of the primal sequence $\{( \pmb{y}^k, \pmb{\xi}^k )\}$. More specifically, for all $k\geq 1$, we have
\begin{align}
\label{eq:primal-convergence}
\norm{\pmb{y}^k-\py^*_\gamma}^2_2+\gamma\norm{\pmb{\xi}^k-\pxi^*_\gamma}_2^2\leq 2\left(p^*_\gamma-g_\gamma(\pmb{\theta}^k)\right)\rightarrow 0.
\end{align}
\end{theorem}
\proof
Let $Q = \big\{ (\pmb{y},\pmb{\xi}):  A_1^{ii}~\pmb{y} +A_2^{ii}~\pmb{\xi} \geq  0,\quad i\in\cK \big\}$. Given $\pmb{\theta}^k\geq 0$ for any $k\geq 1$, since $\cL_\gamma(\pmb{y}, \pmb{\xi}, \pmb{\theta}^k)$ is a quadratic function in $(\pmb{y},\pmb{\xi})$, we can compute $\cL_\gamma(\py^*_\gamma,\pxi^*_\gamma, \pmb{\theta}^k)$ by using second-order Taylor expansion of around $(\pmb{y}^k, \pmb{\xi}^k)$:
\begin{eqnarray*}
\label{eq:taylor-expansion}
\lefteqn{\cL_\gamma(\py^*_\gamma,\pxi^*_\gamma,\pmb{\theta}^k)=}\\
& & \cL_\gamma(\pmb{y}^k, \pmb{\xi}^k, \pmb{\theta}^k)+
\begin{pmatrix}
\grad_{\pmb{y}}\cL_\gamma(\pmb{y}^k, \pmb{\xi}^k, \pmb{\theta}^k)\\
\grad_{\pmb{\xi}}\cL_\gamma(\pmb{y}^k, \pmb{\xi}^k, \pmb{\theta}^k)
\end{pmatrix}^\top
\begin{pmatrix}
\py^*_\gamma-\pmb{y}^k\\
\pxi^*_\gamma-\pmb{\xi}^k
\end{pmatrix}
+\frac{1}{2}
\begin{pmatrix}
\py^*_\gamma-\pmb{y}^k\\
\pxi^*_\gamma-\pmb{\xi}^k
\end{pmatrix}^\top
\begin{pmatrix}
\bI & \mathbf{0}\\
\mathbf{0}^\top & \gamma \mathbf{I}
\end{pmatrix}
\begin{pmatrix}
\py^*_\gamma-\pmb{y}^k\\
\pxi^*_\gamma-\pmb{\xi}^k
\end{pmatrix}.
\end{eqnarray*}
Note that $g_\gamma(\pmb{\theta}^k)=\cL(\pmb{y}^k, \pmb{\xi}^k,\pmb{\theta}^k)$, and since $(\py^*_\gamma, \pxi^*_\gamma)\in Q$, the first-order optimality condition for $( \pmb{y}^k, \pmb{\xi}^k )$ implies that the second term on the right-hand side of the above equality is non-negative. 
Therefore,
\begin{align*}
p_\gamma^* &\geq p_\gamma^* - \sum_{(i,j)\in\cG}
\fprod{\pmb{\theta}^k_{ij},\ A_1^{ij}
\py^*_\gamma +A_2^{ij}
\pxi^*_\gamma} \\ &= \cL_\gamma(\py^*_\gamma, \pxi^*_\gamma, \pmb{\theta}^k) \geq
g_\gamma(\pmb{\theta}^k)+\frac{1}{2}\left(\norm{\py^*_\gamma-\pmb{y}^k}_2^2+\gamma\norm{\pxi^*_\gamma-\pmb{\xi}^k}_2^2\right),
\end{align*}
where the first inequality above follows from $\pmb{\theta}^k\geq 0$ and
$
A_1^{ij}
\py^*_\gamma+A_2^{ij}
\pxi^*_\gamma
\geq 0$ for all 
$(i,j)\in\cG$ -- since $(\py^*_\gamma, \pxi^*_\gamma)$ satisfies all constraints in \eqref{regularize}.\qed
\endproof
\begin{corollary}
\label{cor:rate}
Let $\gamma=0$, and $\{\pmb{\theta}^k\}$ be some dual sequence such that $\pmb{\theta}^k\geq 0$ for $k\geq 1$ and $\lim_{k\in\integers_+}g_0(\pmb{\theta}^k)=p^*_0$, i.e., $p^*_0=\tfrac{1}{2}\norm{\pmb{y}^*-\bar{\pmb{y}}}_2^2$, where $\pmb{y}^*$ is the unique optimal solution defined in \eqref{tikhonov}. Moreover, let $( \pmb{y}^k, \pmb{\xi}^k )$ denote \emph{an} optimal solution to the minimization problem in \eqref{subgradient} when $\pmb{\theta}$ is set to $\pmb{\theta}^k$ for $k\geq 1$. $\pmb{y}^*$ is the unique limit point of the primal sequence $\{\pmb{y}^k\}$. More specifically, we have \vspace*{-1mm}
\begin{align}
\label{eq:primal-convergence-2}
\norm{\pmb{y}^k-\pmb{y}^*}^2_2\leq 2\left(p^*_0-g_0(\pmb{\theta}^k)\right)\rightarrow 0.
\vspace*{-3mm}
\end{align}
\vspace*{-3mm}
\end{corollary}
In the rest of the paper, we design methods based on dual decomposition to solve the convex regression problem in \eqref{original} or its regularized version in \eqref{regularize} when $N$ is large. Suppose $N$ is so large that solving either \eqref{original}, or \eqref{regularize} using IPM is infeasible due to high memory requirements caused by $\cO(N^2)$ shape constraints. In this scenario, using dual decomposition methods, including the methods proposed in this paper, reduces the memory overhead; but, this will come at the cost of considerable increase in the run time if a high-accuracy solution is desired. That being said, in many applications, low-to-moderate-accuracy approximate solutions usually have significant value to the practitioner; this is when dual decomposition based first-order methods become attractive. Therefore, it is important to understand how the approximation quality of iterate sequence $\{(\pmb{y}^k,\pmb{\xi}^k)\}$ changes as the algorithm runs, in order to better asses the trade of between memory requirement and convergence rate of the method chosen.

Our first objective is to study the rate of convergence in more detail. In particular, Corollary~\ref{cor:rate} implies that the projected subgradient method discussed in Section~\ref{sec:subgradient} guarantees $\norm{\pmb{y}^k-\pmb{y}^*}_2^2=\cO(1/\sqrt{k})$ rate. On the other hand, inspired by Nesterov's smoothing for solving structured non-smooth problems in~\cite{Nesterov05_1J}, we can improve the convergence rate. Indeed, combining the result of Lemma~\ref{lem:continuity} with Theorem~\ref{thm:rate} we see that the convergence rate in function values for the smoothed dual problem in \eqref{eq:dual-problem} implies 
$\pmb{y}^k\rightarrow\pmb{y}^*$, and an $\epsilon$-optimal solution $\pmb{y}_\epsilon$, i.e., $\norm{\pmb{y}_\epsilon-\pmb{y}^*}_2^2\leq \epsilon$, can be computed in $\cO(1/\epsilon)$ iterations.

Our second objective is to study the convergence behavior of $\{\pmb{\xi}^k\}$ sequence. As discussed before in Section~\ref{sec:subgradient}, when $\gamma=0$, using the projected subgradient method cannot guarantee the asymptotic feasibility of $\{(\pmb{y}^k, \pmb{\xi}^k)\}$; in particular, although $y^k_\ell\rightarrow y^*_\ell=\hat{f}_N(\bar{x}_\ell)$ for all $\ell\in\cN$, $\{\xi_\ell^k\}\subset\reals^n$ may not converge to a point in $\partial \hat{f}_N(\bar{x}_\ell)$ for some $\ell\in\cN$. This might be an issue to consider when designing algorithms for convex regression, as for some applications having error bounds on how $\{\xi_\ell^k\}$ approximates a subgradient at $\bar{x}_\ell$ might be as important as having error bounds on how $\{y^k_\ell\}$ approximates the function value at $\bar{x}_\ell$. For instance, when the objective is to fit concave utility functions to consumer data, subgradients can be used to infer consumers' marginal utilities.

These two objectives motivate the next section, where we briefly state a first-order algorithm to efficiently solve the smoothed dual problem in \eqref{subgradient}.
\begin{figure}[h]
\begin{framed}
\small
\textbf{Algorithm APG}\big($\theta^0$\big)\\
Iteration 0: Take $ \tilde{\theta}^1=\theta^0,\quad t_1=1$\\
Iteration $k$: ($k \geq 1$) Compute
\begin{enumerate}[label=\arabic*:~]
\item ${\theta^k}\gets  \Pi_{\cQ}\left( \tilde{\theta}^k+\tfrac{1}{L}\grad \rho(\tilde{\theta}^k) \right)$
\item $t_{k+1}\gets ( 1+ \sqrt{ 1+ 4t_k^2} )/2$
\item $\tilde{\theta}^{k+1}\gets \theta^k+ \frac{t_k -1}{t_{k+1}} \left(\theta^k-\theta^{k-1}\right)$
\end{enumerate}
\vspace*{-0.25 cm}
\end{framed}
\vspace*{-0.25 cm}
\caption{Accelerated Proximal Gradient Algorithm}
\label{fig:apg}
\end{figure}
\vspace*{-0.3 cm}
\subsection{Parallel Accelerated Proximal Gradient~(P-APG) Algorithm}
\label{sec:PAPG}
Let $\rho:\mathbb{R}^d\rightarrow\mathbb{R}$ be a concave function such that $\grad \rho$ is Lipschitz continuous on $\mathbb{R}^d$ with constant $L$, and $\cQ\subset\mathbb{R}^d$ be a convex set. Given an initial iterate $\theta^0$, let $\{\theta^k\}$ be the iterate sequence generated using the gradient ascent method as follows: $\theta^{k+1}=\theta^k+\grad \rho(\theta^k)/L$ for $k\geq 0$. According to Corollary~2.1.2 in~\cite{Nesterov04_1B}, the error bound is given by
\begin{align}
\label{eq:gd-rate}
0\leq\rho^*-\rho(\theta^k)\leq \frac{2L}{k+4} \norm{\theta^0-\theta^*}_2^2,
\end{align}
for all $k\geq 1$ and for any $\theta^*\in\argmin\{\rho(\theta):\ \theta\in\cQ\}$, where $\rho^*=\rho(\theta^*)$. On the other hand, the APG algorithm, \cite{Beck09,Tseng08}, displayed in Fig.~\ref{fig:apg} is based on Nesterov's accelerated gradient method~\cite{Nesterov04_1B,Nesterov05_1J}. 
Corollary~3 in~\cite{Tseng08}, and Theorem 4.4 in~\cite{Beck09} show that for all $k\geq 1$ the error bound for APG is given by
\begin{align}
\label{eq:apg-rate}
0\leq\rho^*-\rho(\theta^k)\leq \frac{2L}{(k+1)^2} \norm{\theta^0-\theta^*}_2^2,
\end{align}
where $\theta^0$ is the initial APG iterate and $\theta^*\in\argmin_{\theta\in\cQ} \rho(\theta)$. Hence, using APG one can compute an $\delta$-optimal solution within at most $\cO(\sqrt{L/\delta})$ APG iterations. Next, we will customize APG algorithm for solving 
\eqref{regularize} when $\gamma>0$.
\begin{definition}
\label{def:C}
Let $A_3$ and $A_4$ denote the matrices formed by vertically concatenating $A_1^{ij}$ and $A_2^{ij}$, respectively, for all 
$(i,j)\in\cG$. Define  
$C\triangleq\begin{bmatrix}  A_3 & A_4 
\end{bmatrix}$, the decision variable vector 
$\pmb{\eta}^{\mathsf{T}}\triangleq\begin{bmatrix} \pmb{y}^{\mathsf{T}} & \pmb{\xi}^{\mathsf{T}} \end{bmatrix}$,
and the following elements related to the regularized problem in \eqref{regularize}. For $i\in\cK$, $\cQ_i=\{(\py_i,\pxi_i):\ \bar{A}_1^{ii}~\pmb{y}_i +\bar{A}_2^{ii}~\pmb{\xi}_i \geq  0\}$ and \vspace*{-3mm}
\begin{align*}
Q \triangleq \big\{\pmb{\eta}=(\py,\pxi):\ (\py_i,\pxi_i)\in\cQ_i,\ i\in\cK\big\}.
\end{align*}
\end{definition}
Now, consider the equivalent representation of \eqref{regularize}:
\begin{align}
\label{omega_primal}
\min_{\pmb{\eta}\in Q} \quad \frac{1}{2} \norm{\pmb{y} -  \bar{\pmb{y}}}_2^2 +  \frac{\gamma}{2} \left\| \pmb{\xi} \right\| _2^2 \quad \text{s.t.} \quad  &C \,\pmb{\eta} \geq 0.
\end{align}
The objective function in \eqref{subgradient} for the dual problem in \eqref{eq:dual-problem}, i.e.,  $\max\{g_\gamma(\pmb{\theta}):\ \pmb{\theta}\geq 0\}$, can be written as
\begin{equation}
\label{eq:g_gamma}
g_\gamma(\pmb{\theta}) = \min\limits_{\pmb{\eta} \in Q }  \left\{ \frac{1}{2} \big\| \pmb{y} -  \bar{\pmb{y}} \big\| _2^2 + \frac{\gamma}{2} \left\| \pmb{\xi} \right\| _2^2 - \fprod{\pmb{\theta}, C \, \pmb{\eta}} \right\}.
\end{equation}
Theorem 7.1 in~\cite{nesterov2005excessive} and Danskin's theorem imply that
\begin{align}
\label{dual gradient}
\nabla g_\gamma(\pmb{\theta}) = - C \, \pmb{\eta}(\pmb{\theta}) ,
\end{align}
where $\pmb{\eta}(\pmb{\theta})$ is the unique minimizer in \eqref{eq:g_gamma}, and $\grad g_\gamma(\pmb{\theta})$  is Lipschitz continuous with constant $L_\gamma$ in~\eqref{eq:lischitz}, where $\norm{C}$ denotes the spectral norm of $C$.
\begin{align}
\label{eq:lischitz}
L_\gamma = \frac{1}{\gamma}~\sigma_{\max}^2(C)=\frac{1}{\gamma}~\norm{C}^2.
\end{align}
\begin{figure}[thpb]
\begin{framed}
{
\small
\textbf{Algorithm P-APG}\big($\gamma, \pmb{\theta}^0$\big) \\
Iteration 0: Set $\tilde{\pmb{\theta}}^1 = \pmb{\theta}^0, t_1=1 $ and $L_\gamma = \frac{1}{\gamma}~\sigma_{\max}^2(C)$\\
Iteration $k$: ($k \geq 1$) Compute
\vspace{-0.2cm}
\begin{enumerate}[label=\arabic*:~]
\item $\pmb{\eta}^k \gets \argmin\limits_{\pmb{\eta}\in Q }\left\{ \frac{1}{2} \big\| \pmb{y} -  \bar{\pmb{y}} \big\| _2^2 + \frac{\gamma}{2} \left\| \pmb{\xi} \right\| _2^2 - \fprod{C^\top\tilde{\pmb{\theta}}^k, \pmb{\eta}} \right\}$
\item $\pmb{\theta}^k\gets  \left( \tilde{\pmb{\theta}}^k-\frac{1}{L_\gamma}C\pmb{\eta}^k \right)_+$
\item $t_{k+1}\gets ( 1+ \sqrt{ 1+ 4t_k^2} )/2$
\item $\tilde{\pmb{\theta}}^{k+1}\gets \pmb{\theta}^k+ \frac{t_k -1}{t_{k+1}} \left(\pmb{\theta}^k-\pmb{\theta}^{k-1}\right)$
\end{enumerate}
\vspace{-0.25cm}
}
\end{framed}
\vspace{-0.25cm}
\caption{ Parallel APG Algorithm~(P-APG)}
\label{fig:papg}
\end{figure}

Parallel~APG algorithm~(P-APG), displayed in Fig.~\ref{fig:papg}, is the customized version of APG algorithm in Fig.~\ref{fig:apg} to solve \eqref{eq:dual-problem}. Note that the computation in Step-1 can be carried out in \emph{parallel} using $K$ processors, each solving a small-size QP. Later in Section~\ref{sec:complexity}, we discuss the computational complexity of one P-APG iteration in detail.
\paragraph{Adaptive Step Size Strategy:}One important property of APG methods is the ability to adopt an adaptive step-size sequence. Note $L_\gamma$, the Lipschitz constant of $\grad g_\gamma(\theta)$, may not be known in advance or may be too conservative in practice -- leading to very small steps. Instead of constant step size $1/L_\gamma$ in Step-2~of P-APG, if one uses an adaptive step sequence $\{1/s_k\}$, the $\cO(L_\gamma/k^2)$ rate shown in~\cite{Beck09} still holds as long as
\begin{align}
\label{eq:adaptive-check}
 g_\gamma (\pmb{\theta}^k) \geq g_\gamma(\tilde{\pmb{\theta}}^k)+\fprod{\grad g_\gamma(\tilde{\pmb{\theta}}^k),~  \pmb{\theta}^k- \tilde{\pmb{\theta}}^k }-\frac{s_k}{2}\norm{ \pmb{\theta}^k - \tilde{\pmb{\theta}}^k }_2^2,
\end{align}
holds for all $k$ where $\pmb{\theta}^k$ is computed using $s_k$ instead of $1/L_\gamma$. Clearly, one can choose $s^k \leq L_\gamma$; possibly take longer steps compared to constant step size $1/L_\gamma$ and still has a convergence guarantee with the same rate. We adopted the following rule in our numerical tests: let $\upsilon>1$, for $k\geq 1$ we set $s_{k}=s_{k-1} \upsilon^{\ell_k-1}$ where $\ell_k\geq 0$ is the smallest integer such that \eqref{eq:adaptive-check} holds, and $s_{0}=L_\gamma$.

In the rest of the paper, other than the numerical section, for the sake of simplicity we assume $s_k=L_\gamma$ for all $k$. 
To better understand the convergence rate of P-APG, 
next, we provide a bound on $\norm{\pmb{\theta}^*_\gamma}_2$ for all $\pmb{\theta}^*_\gamma\in\Theta^*_\gamma$.
\begin{lemma}
\label{lem:dual-bound}
Given $\gamma\geq 0$ and $\delta\geq 0$, let $\pmb{\theta}_{\gamma,\delta}\geq 0$ be a $\delta$-optimal solution to \eqref{eq:dual-problem}, i.e., $0\leq p_\gamma^*-g_\gamma(\pmb{\theta}_{\gamma,\delta})\leq \delta$. Given $\{(\bar{x}_\ell,\bar{y}_\ell)\}_{\ell\in\cN}$, for all $\ell\in\cN$, define $\tilde{y}_\ell\triangleq\hat{y}+\tfrac{\alpha}{2}\norm{\bar{x}_\ell-\hat{x}}_2^2$ for some given $\alpha>0$, where $\hat{x}\triangleq\frac{1}{N}\sum_{\ell\in\cN}\bar{x}_\ell$ and $\hat{y}\triangleq\frac{1}{N}\sum_{\ell\in\cN}\bar{y}_\ell$. Then
\begin{align}
\label{eq:theta-norm-bound}
\norm{\pmb{\theta}_{\gamma,\delta}}_1\leq \frac{2}{\alpha\upsilon}\Big(\delta-p^*_\gamma+\tfrac{1}{2}\sum_{\ell\in\cN}\left(\tilde{y}_\ell-\bar{y}_\ell\right)^2+\gamma\alpha^2\norm{\bar{x}_\ell-\hat{x}}_2^2\Big)\triangleq B(\gamma,\delta,\alpha),
\end{align}
where $\upsilon\triangleq\min\limits_{(i,j)\in\cG}
\{\norm{\bar{x}_{\ell_1}-\bar{x}_{\ell_2}}_2^2:\ \ell_1\in\cC_i,\ \ell_2\in\cC_j\}$.
\end{lemma}
\proof
For given $\alpha>0$, define $h:\reals^n\rightarrow\reals$ such that $h(x)\triangleq\hat{y}+\tfrac{\alpha}{2}\norm{x-\hat{x}}_2^2$. Note that for all $\ell\in\cN$, we have $\tilde{y}_\ell=h(\bar{x}_\ell)$, and $\tilde{\xi}_\ell\triangleq\grad h(\bar{x}_\ell)=\alpha(\bar{x}_\ell-\hat{x})$. Since $h$ is strongly convex with modulus $\alpha>0$, for any $(\ell_1,\ell_2)\in\cN\times\cN$, it follows that
\begin{align}
\label{eq:slater-point}
\tilde{y}_{\ell_2}-\tilde{y}_{\ell_1}+\fprod{\tilde{\xi}_{\ell_1},~ \bar{x}_{\ell_1}-\bar{x}_{\ell_2}}\geq\frac{\alpha}{2}\norm{\bar{x}_{\ell_2}-\bar{x}_{\ell_1}}_2^2\geq 0.
\end{align}
Let $\tilde{\pmb{\eta}}=[\tilde{\pmb{y}}^\top \tilde{\pmb{\xi}}^\top]^\top$ such that $\tilde{\pmb{y}}=[\tilde{\pmb{y}}_i]_{i\in\cK}$ and $\tilde{\pmb{\xi}}=[\tilde{\pmb{\xi}}_i]_{i\in\cK}$, where $\tilde{\pmb{y}}_i=[\tilde{y}_\ell]_{\ell\in\cC_i}$ and $\tilde{\pmb{\xi}}_i=[\tilde{\xi}_\ell]_{\ell\in\cC_i}$. Hence, $\tilde{\pmb{\eta}}\in Q$ is a Slater point for the problem in \eqref{regularize}, or equivalently \eqref{omega_primal}. Since
$C\tilde{\pmb{\eta}}\geq \frac{\alpha\upsilon}{2} \mathbf{1}> 0$, 
it follows from \eqref{eq:g_gamma} that
\begin{equation}
\frac{\alpha\upsilon}{2}\norm{\pmb{\theta}_{\gamma,\delta}}_1\leq\fprod{\pmb{\theta}_{\gamma,\delta},~C \tilde{\pmb{\eta}}}\leq \frac{1}{2}\norm{\tilde{\pmb{y}} -  \bar{\pmb{y}}}_2^2 +  \frac{\gamma}{2} \left\| \tilde{\pmb{\xi}} \right\| _2^2-g_\gamma(\pmb{\theta}_{\gamma,\delta}),
\end{equation}
and the result follows from $\delta$-optimality, i.e., $p_\gamma^*-g_\gamma(\pmb{\theta}_{\gamma,\delta})\leq \delta$.\qed
\endproof
\begin{remark}
\label{rem:dual_bounds}
When $\gamma=0$, for any $\pmb{\theta}^*_0\in\Theta^*_0$, it follows that
$$ \norm{\pmb{\theta}^*_0}_1\leq B(0,0,\alpha)=\frac{2}{\alpha\upsilon}\left(\tfrac{1}{2}\sum_{\ell\in\cN}(\tilde{y}_\ell-\bar{y}_\ell)^2-p^*_0\right)\leq \frac{1}{\alpha\upsilon}\sum_{\ell\in\cN}(\tilde{y}_\ell-\bar{y}_\ell)^2\triangleq B_\theta(\alpha).$$
Note $p_0^*\leq p_\gamma^*$ for all $\gamma\geq 0$; hence, when $\gamma>0$, for any $\pmb{\theta}^*_\gamma\in\Theta^*_\gamma$, it follows that
\begin{align*}
\norm{\pmb{\theta}^*_\gamma}_1\leq B(\gamma,0,\alpha)
&\leq B(0,0,\alpha)+\frac{\gamma\alpha}{\upsilon}\sum_{\ell\in\cN}\norm{\bar{x}_\ell-\hat{x}}_2^2\\
&\leq B_\theta(\alpha)+\frac{\gamma\alpha}{\upsilon}\sum_{\ell\in\cN}\norm{\bar{x}_\ell-\hat{x}}_2^2\triangleq B_\theta(\gamma,\alpha).
\end{align*}
\end{remark}
The bound on $\norm{\pmb{\theta}_{\gamma,\delta}}_1$ given in \eqref{eq:theta-norm-bound} holds for all $\alpha>0$. Therefore, by choosing $\alpha>0$ depending on $\gamma\geq 0$, we optimize the upper bounds $B_\theta(\alpha)$ and $B_\theta(\gamma,\alpha)$ defined in Remark~\ref{rem:dual_bounds}.
\begin{lemma}
\label{lem:opt-alpha}
Given $\gamma\geq 0$, let $\alpha^*_\gamma\triangleq\argmin\{B_\theta(\gamma,\alpha):\ \alpha>0\}$, and $\alpha^*\triangleq \alpha^*_0$ for $\gamma=0$, i.e., $\alpha^*=\argmin\{B_\theta(\alpha):\ \alpha>0\}$. For any $\gamma\geq 0$, $\alpha^*_\gamma$ can be computed as follows
\begin{equation}\label{eq:alpha-opt}
\alpha^*_\gamma=4~\left( \frac{\sum_{\ell\in\cN}(\tilde{y}_\ell-\hat{y})^2}{\sum_{\ell\in\cN}\norm{\bar{x}_\ell-\hat{x}}_2^2\big(\norm{\bar{x}_\ell-\hat{x}}_2^2+8\gamma\big)}\right)^{1/2},
\end{equation}
leading to tight upper bounds $B^*_\theta\triangleq B_\theta(\alpha^*)$ and $B^*_\theta(\gamma)\triangleq B_\theta(\gamma,\alpha^*_\gamma)$.
\end{lemma}
\begin{proof}
According to definition of $\{\tilde{y}_\ell\}_{\ell\in\cN}$ given in Lemma~\ref{lem:dual-bound}, $B_\theta(\gamma,\alpha)$ can be explicitly stated as follows:
\begin{equation}\label{eq:explicit-B-tetha-gamma}
    B_\theta(\gamma,\alpha)=\frac{1}{\alpha\upsilon}\sum_{\ell\in\cN}\Big(\hat{y}-\bar{y}_\ell+\tfrac{\alpha}{2}\norm{\hat{x}-\bar{x}_\ell}_2^2\Big)^2+\frac{\gamma\alpha}{\upsilon}\sum_{\ell\in\cN}\norm{\hat{x}-\bar{x}_\ell}_2^2.
\end{equation}
To simplify the notation, let $p_\ell\triangleq\hat{y}-\bar{y}_\ell$, and $q_\ell\triangleq\tfrac{1}{2}\norm{\hat{x}-\bar{x}_\ell}_2^2$ for $\ell\in\cN$. Via the change of variables $\beta=\sqrt{\alpha}$, we obtain the following equivalent problem:
\begin{equation}
\min_\beta\Big\{ w(\beta)\triangleq\sum_{\ell\in\cN}\Big(\frac{1}{\beta}~p_\ell+\beta~q_\ell\Big)^2+2\gamma\beta^2\sum_{\ell\in\cN}q_\ell:\  \beta>0\Big\}.
\end{equation}
Clearly, we have
\begin{equation}
w'(\beta)=\sum_{\ell\in\cN}q_\ell(q_\ell+4\gamma)\beta-\frac{p_\ell^2}{\beta^3},\quad w''(\beta)=\sum_{\ell\in\cN}q_\ell(q_\ell+4\gamma)+3\frac{p_\ell^2}{\beta^4}.
\end{equation}
Since $w''(\beta)\geq 0$ for $\beta>0$, $w(\beta)$ is a convex function and first-order necessary optimality condition, i.e., $w'(\beta^*)=0$, is also sufficient. In particular, solving for $\beta^*$ and setting $\alpha^*_\gamma=\sqrt{\beta^*}$ gives the desired result in \eqref{eq:alpha-opt}.\qed
\end{proof}
Let constants $B^*_\theta$ and $B^*_\theta(\gamma)$ be as defined in Lemma~\ref{lem:opt-alpha}. Now, using \eqref{eq:lischitz} and the bounds given in Remark~\ref{rem:dual_bounds}, we can customize the generic rate results in \eqref{eq:gd-rate} for \emph{gradient ascent} and those in \eqref{eq:apg-rate} for APG methods. In particular, for any $\gamma>0$, in order to compute a $\delta$-optimal solution to the problem in \eqref{eq:dual-problem}, i.e., $\pmb{\theta}_{\gamma,\delta}\geq 0$ such that $0\leq p_\gamma^*-g_\gamma(\pmb{\theta}_{\gamma,\delta})\leq \delta$, the gradient ascent method requires $\cO(L_\gamma/\delta)=\cO({B^*_\theta}^2/(\gamma \, \delta))$ iterations.
On the other hand, 
P-APG in Fig.~\ref{fig:papg} can compute a $\delta$-optimal solution to \eqref{eq:dual-problem} within $\cO(\sqrt{L_\gamma/\delta})$ iterations. More precisely, \eqref{eq:lischitz} implies $\cO(B^*_\theta(\gamma)/(\gamma \, \delta)^{1/2})$ iteration complexity for P-APG when applied to \eqref{eq:dual-problem}.

The $\cO(1)$ constant depends on $\sigma_{\max}(C)$, and to better have a better understanding of how it grows with the problem size, we provide some bounds for $\sigma_{\max}(A_1)$, $\sigma_{\max}(A_2)$, and $\sigma_{\max}(C)$.
\begin{lemma}
\label{lem:O1-constant}
Let $A_1\in\mathbb{R}^{N(N-1)\times N}$ and $A_2\in\mathbb{R}^{N(N-1)\times Nn}$ be the matrices in \eqref{original_compact}, i.e., corresponding to the constraints in \eqref{original}; and let $A_3$, $A_4$, and $C$ be the matrices as given in Definition~\ref{def:C}. Then, $\sigma_{\max}(A_3)\leq\sigma_{\max}(A_1)=\sqrt{2N}$, $\sigma_{\max}(A_4)\leq\sigma_{\max}(A_2)\leq B_x N$, and $\sigma_{\max}(C)\leq \sqrt{2N}+B_x N$.
\end{lemma}
\begin{proof}
It is easy to observe that $A_1^\top A_1=2 \Omega$, where $\Omega\in\reals^{N\times N}$ denotes the Laplacian matrix of a complete graph with $N$ vertices, i.e., for each $i=1,\ldots,N$, $\Omega_{ii}=N-1$, and $\Omega_{ij}=-1$ for all $j\neq i$. It is known that $\Omega$ has two distinct eigenvalues: 0 (with multiplicity $1$) and $N$ (with multiplicity $N-1$). Therefore, $\sigma_{\max}(A_1)=\sqrt{2N}$; and since $A_3$ is a submatrix of $A_1$, one immediately has $\sigma_{\max}(A_3)\leq\sigma_{\max}(A_1)$.

Since $A_2$ is block-diagonal, we have $\sigma_{\max}(A_2)=\max_{\ell\in\cN}\{\sigma_{\max}(X_\ell)\}$, where $X_\ell=-T_\ell \bar{X}$ (see Definition~\ref{def:A1A2} for $T_\ell$ and $\bar{X}$). Hence, we have $\sigma_{\max}(A_2)\leq\norm{\bar{X}}~\max\{\norm{T_\ell}:~\ell\in\cN\}$. For $\ell\in\cN$, let $\Omega_\ell=T_\ell^\top T_\ell$; it is easy to observe that $\Omega_\ell\in\reals^{N\times N}$ is the Laplacian matrix of a star-tree with $N-1$ leaves ($\ell$ is the internal node). It is known that $\Omega$ has three distinct eigenvalues: 0 (with multiplicity $1$), $1$ (with multiplicity $N-2$), and $N$ (with multiplicity 1). Therefore, $\norm{T_\ell}=\sqrt{N}$ for all $\ell\in\cN$. On the other hand, since $\norm{\bar{x}_\ell}_2\leq B_x$ for $\ell\in\cN$, $\norm{\bar{X}}\leq B_x\sqrt{N}$. Therefore, $\sigma_{\max}(A_2)\leq B_x N$; and since $A_4$ is a submatrix of $A_2$, one immediately has $\sigma_{\max}(A_4)\leq\sigma_{\max}(A_2)$. Finally, since $C=[A_3 A_4]$, clearly $\norm{C}\leq\norm{A_3}+\norm{A_4}$.\qed
\end{proof}

Next, we study the error bounds for inexact solutions. Given $\gamma>0$, let $\pmb{\theta}_{\gamma,\delta}$ be a $\delta$-optimal solution to \eqref{eq:dual-problem}, and $(\pmb{y}_{\gamma,\delta}, \pmb{\xi}_{\gamma,\delta})$ be the optimal solution to the minimization problem in \eqref{subgradient}, or equivalently to \eqref{eq:g_gamma}, when $\pmb{\theta}$ is set to $\pmb{\theta}_{\gamma,\delta}$. In Theorem~\ref{bound} 
we establish error bounds on the suboptimality $\norm{\pmb{y}_{\gamma,\delta}-\pmb{y}^*}_2$, and on the infeasibility $\norm{ (A_1 \pmb{y}_{\gamma,\delta} + A_2 \pmb{\xi}_{\gamma,\delta} )_{-} }_2$.   
\begin{theorem}
\label{bound}
Given $\gamma>0$, let $\left(\pmb{y}^*_\gamma, \pmb{\xi}^*_\gamma\right)$ and $\pmb{\theta}^*_\gamma$ denote the optimal solutions to \eqref{regularize}  and \eqref{eq:dual-problem}, respectively. Let $\pmb{\theta}_{\gamma,\delta}$ be a $\delta$-optimal solution to \eqref{eq:dual-problem}, and $\left(\pmb{y}_{\gamma,\delta}, \pmb{\xi}_{\gamma,\delta}\right)$ be the minimizer in 
\eqref{eq:g_gamma} when $\pmb{\theta}$ is set to $\pmb{\theta}_{\gamma,\delta}$. 
For all $\gamma,\delta > 0$, the following bounds hold: 
\begin{flalign}
\label{bound ineq}
&\norm{ \pmb{y}_{\gamma,\delta} - \pmb{y}^*}_2 \leq \norm{\pmb{\xi}^*}_2\sqrt{\gamma}+\sqrt{2\delta},\\
&\norm{ (A_1 \, \pmb{y}_{\gamma,\delta} + A_2 \, \pmb{\xi}_{\gamma,\delta} )_{-} }_2 \leq 
2\sqrt{N \delta} + 
B_x N \sqrt{\frac{2 \delta }{\gamma}}. \label{bound vio}
\end{flalign}
Moreover, both starting from the initial iterate $\pmb{\theta}^0=\mathbf{0}$, P-APG can compute $(\pmb{y}_{\gamma,\delta},\pmb{\xi}_{\gamma,\delta})$ within $K(\delta,\gamma)=\sigma_{\max}(C)B^*_\theta(\gamma)\sqrt{2/(\gamma\delta)}$ iterations while gradient ascent 
requires $2(B^*_\theta\sigma_{\max}(C))^2/(\gamma\delta)$ iterations, where $\sigma_{\max}(C)=\cO(N)$.
\end{theorem}
\proof
Given $\pmb{\theta}_{\gamma,\delta}\geq 0$ and the corresponding minimizer, $\left(\pmb{y}_{\gamma,\delta}, \pmb{\xi}_{\gamma,\delta}\right)$, to the problem in \eqref{eq:g_gamma} when $\pmb{\theta}$ is set to $\pmb{\theta}_{\gamma,\delta}$, Theorem~\ref{thm:rate} implies that
\begin{equation}
\label{eq:delta-bound}
\norm{\pmb{y}_{\gamma,\delta}-\pmb{y}^*_\gamma}^2_2+\gamma\norm{\pmb{\xi}_{\gamma,\delta}-\pmb{\xi}^*_\gamma}_2^2\leq 2\left(p^*_\gamma-g_\gamma(\pmb{\theta}_{\gamma,\delta})\right)\leq2\delta.
\end{equation}
Hence, Lemma~\ref{lem:continuity} and \eqref{eq:delta-bound} together imply that
\begin{equation*}
\norm{\pmb{y}_{\gamma,\delta}-\pmb{y}^*}_2\leq\norm{\pmb{y}_{\gamma,\delta}-\pmb{y}^*_\gamma}_2+\norm{\pmb{y}^*_\gamma-\pmb{y}^*}_2\leq\norm{\pmb{\xi}^*}_2\sqrt{\gamma}+\sqrt{2\delta}.
\end{equation*}
Moreover, since $\left(\pmb{y}_{\gamma,\delta}, \pmb{\xi}_{\gamma,\delta}\right)\in Q$ and $\left(\pmb{y}^*_{\gamma}, \pmb{\xi}^*_{\gamma}\right)$ is feasible to \eqref{regularize}, i.e., $A_1\pmb{y}^*_\gamma+A_2\pmb{\xi}^*_\gamma\geq 0$, we have
\begin{eqnarray}
\lefteqn{\norm{ (A_1 \, \pmb{y}_{\gamma,\delta} + A_2 \, \pmb{\xi}_{\gamma,\delta} )_{-} }_2} \nonumber\\
& & = \norm{ (A_3 \, \pmb{y}_{\gamma,\delta} + A_4 \, \pmb{\xi}_{\gamma,\delta} )_{-} }_2=\norm{ (A_3 \, \pmb{y}_{\gamma,\delta} + A_4 \, \pmb{\xi}_{\gamma,\delta} )_{-}-(A_3 \, \pmb{y}^*_{\gamma} + A_4 \, \pmb{\xi}^*_{\gamma} )_{-} }_2, \nonumber \\
& & \leq\norm{ A_3 \, (\pmb{y}_{\gamma,\delta}-\pmb{y}^*_{\gamma}) + A_4 \, (\pmb{\xi}_{\gamma,\delta}-\pmb{\xi}^*_{\gamma} )}_2 \nonumber \\
& &\leq\sigma_{\max}(A_3)\norm{\pmb{y}_{\gamma,\delta}-\pmb{y}^*_\gamma}_2+\sigma_{\max}(A_4)\norm{\pmb{\xi}_{\gamma,\delta}-\pmb{\xi}^*_\gamma}_2
\label{eq:infeasibility-bound}
\end{eqnarray}
where the first inequality follows from the fact that $\norm{\pmb{x}-\pmb{y}}_2\geq\norm{(\pmb{x})_{-}-(\pmb{y})_{-}}_2$ for any $\pmb{x}$ and $\pmb{y}$. The infeasibility result in \eqref{bound vio} immediately follows from \eqref{eq:delta-bound} and \eqref{eq:infeasibility-bound}. The iteration complexity bounds can be obtained using the arguments immediately after Remark~\ref{rem:dual_bounds}.\qed
\endproof

As for some applications having an error bound on how $\pmb{\xi}_{\gamma,\delta}$ approximates $\pmb{\xi}^*$, i.e., the subgradients at $\{\bar{x}_\ell\}_{\ell\in\cN}$, is crucial. Next, we 
show that $\norm{\pmb{\xi}_{\gamma,\delta}-\pmb{\xi}^*}_2$ is indeed small. 

\begin{theorem}
\label{thm:xi-bound}
There exists $K>0$ such that $\norm{\pmb{\xi}^*_\gamma-\pmb{\xi}^*}_2\leq K\norm{A_1(\pmb{y}^*_\gamma-\pmb{y}^*)}_2$; hence, $\norm{\pmb{\xi}^*_\gamma-\pmb{\xi}^*}_2\leq K\sigma_{\max}(A_1)\norm{\pmb{\xi}^*}_2\sqrt{\gamma}$, which implies
\begin{equation*}
\norm{\pmb{\xi}_{\gamma,\delta}-\pmb{\xi}^*}_2\leq  K\norm{\pmb{\xi}^*}_2\sqrt{2N\gamma}+\sqrt{\frac{2\delta}{\gamma}}.
\end{equation*}
\end{theorem}
\proof
Since $\pmb{y}^*$ is the unique optimal solution to \eqref{original_compact}, \eqref{tikhonov} implies that $\pmb{\xi}^*=\argmin\{\norm{\pmb{\xi}}_2:\ A_1 \, \pmb{y}^*+A_2 \, \pmb{\xi}\geq \pmb{0}\}$. Similarly, \eqref{regularize} implies that $\pmb{\xi}^*_\gamma=\argmin\{\norm{\pmb{\xi}}_2:\ A_1 \, \pmb{y}^*_\gamma+A_2 \, \pmb{\xi}\geq \pmb{0}\}$. Define $\pmb{h}(\gamma)\triangleq-A_1\pmb{y}^*_\gamma$ for $\gamma\geq 0$. Note from Lemma~\ref{lem:continuity}, we have $\pmb{h}(0)=-A_1\pmb{y}^*_0=-A_1\pmb{y}^*$. Therefore, for $\gamma\geq 0$,
\begin{equation}
\pmb{\xi}^*_\gamma=\argmin \big\{\norm{\pmb{\xi}}_2:\ A_2 \, \pmb{\xi}\geq \pmb{h}(\gamma) \big\}.
\end{equation}
Note that for $\gamma=0$, $\pmb{\xi}^*_0=\pmb{\xi}^*$. Sensitivity of metric projection onto parametric polyhedral sets is studied in~\cite{Yen95_1J}. According to Theorem~2.1 in~\cite{Yen95_1J}, there exists $K>0$ such that
\begin{equation}
\norm{\pmb{\xi}^*_\gamma-\pmb{\xi}^*_{\gamma'}}_2\leq K\norm{\pmb{h}(\gamma)-\pmb{h}(\gamma')}_2\leq K\sigma_{\max}(A_1)\norm{\pmb{y}^*_\gamma-\pmb{y}^*_{\gamma'}}_2,\quad \forall\ \gamma,\gamma'\geq 0.
\end{equation}
Therefore, given $\gamma>0$, setting $\gamma'=0$, and using Lemma~\ref{lem:continuity}, we have
\begin{equation}
\label{eq:sensitivity-bound}
\norm{\pmb{\xi}^*_\gamma-\pmb{\xi}^*}_2\leq K\sigma_{\max}(A_1)\norm{\pmb{y}^*_\gamma-\pmb{y}^*}_2\leq K\sigma_{\max}(A_1)\norm{\pmb{\xi}^*}\sqrt{\gamma}.
\end{equation}
Moreover, \eqref{eq:delta-bound} implies that $\norm{\pmb{\xi}_{\gamma,\delta}-\pmb{\xi}^*_\gamma}_2\leq\sqrt{\frac{2\delta}{\gamma}}$. Hence, combining this with \eqref{eq:sensitivity-bound} gives the desired result since $\sigma_{\max}(A_1)=\sqrt{2N}$.\qed
\endproof

We can summarize Theorem~\ref{bound} and Theorem~\ref{thm:xi-bound} briefly as follows. If the main objective is the function value approximation and estimating the subgradients are not crucial, then according to 
Theorem~\ref{bound}, for any given $\epsilon>0$, setting $\gamma=\delta=\epsilon$ implies that $\pmb{y}_{\gamma,\delta}\in\reals^{N}$ satisfies $\norm{\pmb{y}_{\gamma,\delta}-\pmb{y}^*}_2^2=\cO(\epsilon)$ and it can be computed within $\cO(N^2 {B^*_\theta}^2/\epsilon^2)$ iterations of the gradient ascent method on \eqref{eq:dual-problem} (which is the \emph{same} as the iteration complexity of the projected subgradient method applied to \eqref{eq:dual-problem} for $\gamma=0$), and within $\cO(N B^*_\theta(\epsilon)/\epsilon)$ iterations of P-APG in Fig.~\ref{fig:papg} on \eqref{eq:dual-problem}. On the other hand if the subgradient approximation is important too, then according to Lemma~\ref{lem:continuity}, Theorem~\ref{bound} and Theorem~\ref{thm:xi-bound}, for any given $\epsilon>0$, by setting $\gamma=\epsilon$ and $\delta=\epsilon^2$ implies that $\pmb{y}_{\gamma,\delta}\in\reals^{N}$ satisfies $\norm{\pmb{y}_{\gamma,\delta}-\pmb{y}^*}_2^2=\cO(\epsilon)$, $\norm{\pmb{\xi}_{\gamma,\delta}-\pmb{\xi}^*}_2^2=\cO(\epsilon)$ and $\norm{(A_1\pmb{y}_{\gamma,\delta}+A_2\pmb{\xi}_{\gamma,\delta})_-}_2^2\leq \cO(\epsilon)$ within $\cO(N^2 {B^*_\theta}^2/\epsilon^3)$ iterations using the gradient ascent method on \eqref{eq:dual-problem}, and within $\cO(N B^*_\theta(\epsilon)/\epsilon^{3/2})$ iterations using P-APG in Fig.~\ref{fig:papg} on \eqref{eq:dual-problem}.
\subsubsection{Continuation Method for Convex Regression}
\label{sec:cont-method}
Let $\pmb{\theta}_{\gamma,\delta}$ be a $\delta$-optimal solution to \eqref{eq:dual-problem}, and $\left(\pmb{y}_{\gamma,\delta}, \pmb{\xi}_{\gamma,\delta}\right)$ be the minimizer in \eqref{eq:g_gamma} when $\pmb{\theta}$ is set to $\pmb{\theta}_{\gamma,\delta}$. In Section~\ref{sec:PAPG}, we have seen that for any fixed $\epsilon$, setting $\gamma=\delta=\epsilon$ implies that $\pmb{y}_{\gamma,\delta}$ can be computed within $\cO(N B^*_\theta(\epsilon)/\epsilon)$ iterations of P-APG and it satisfies $\norm{\pmb{y}_{\gamma,\delta}-\pmb{y}^*}_2^2=\cO(\epsilon)$. In this section, we describe a continuation method to solve \eqref{original_compact}. In particular, we would like to generate an iterate sequence $\{\py^t\}_{t\in\integers_+}$ such that $\py^{(t)}\rightarrow\py^*$ as $t\rightarrow +\infty$ with the following properties:\\
i) for any $\epsilon>0$, $\py^{(t)}$ satisfies $\norm{\pmb{y}^{(t)}-\pmb{y}^*}_2^2=\cO(\epsilon)$ for all $t\geq T_\epsilon=\cO(\log(1/\epsilon))$;\\
ii) moreover, $T_\epsilon$  iterations of the continuation require at most $\cO(1/\epsilon)$ P-APG iterations in total, i.e., the algorithm generates an asymptotically optimal iterate sequence with $\cO(1/\epsilon)$ rate without fixing the algorithmic parameters depending on the tolerance $\epsilon>0$.

Let $\beta>1$ and define $\{\epsilon_t\}_{t\in\integers_+}$ such that $\epsilon_t=\epsilon_0/\beta^t$ for some $\epsilon_0>0$. Also define $\{\gamma_t\}_{t\in\integers_+}$ and $\{\delta_t\}_{t\in\integers_+}$ such that $\gamma_t=\kappa_\gamma \epsilon_t$ and $\delta_t=\kappa_\delta \epsilon_t$ for $t\geq 1$ for some $\kappa_\gamma, \kappa_\delta>0$. Next, for all $t\geq 1$, let $\pmb{\theta}^{(t)}\triangleq\pmb{\theta}_{\gamma_t,\delta_t}$ be a $\delta_t$-optimal solution to \eqref{eq:dual-problem} when $\gamma=\gamma_t$, such that it is computed using P-APG in Fig.~\ref{fig:papg} starting from the initial iterate $\pmb{\theta}^{(t-1)}$, where $\pmb{\theta}^{(0)}=\mathbf{0}$, and $\left(\pmb{y}^{(t)}, \pmb{\xi}^{(t)}\right)$ be the minimizer in \eqref{eq:g_gamma} when $\pmb{\theta}$ is set to $\pmb{\theta}^{(t)}$ and $\gamma=\gamma_t$. Then clearly from \eqref{bound ineq}, we have
\begin{equation}
\norm{\py^{(t)}-\py^*}_2\leq\norm{\pxi^*}_2\sqrt{\gamma_t}+\sqrt{2\delta_t}=\Gamma\beta^{-\tfrac{t}{2}},\quad t\geq 1,
\end{equation}
where $\Gamma\triangleq\sqrt{\epsilon_0}\left(\norm{\pxi^*}_2\sqrt{\kappa_\gamma}+\sqrt{2\kappa_\delta}\right)$. Therefore, $\norm{\py^{(t)}-\py^*}_2^2\leq\epsilon$ for all $t\geq T_\epsilon\triangleq\lceil\log_\beta(\Gamma^2/\epsilon)\rceil$. Let $\py_\epsilon\triangleq\py^{(T_\epsilon)}$; hence, $\norm{\py_\epsilon-\py^*}_2^2\leq\epsilon$.

Note that for all $t\geq 1$, starting from $\pt^{(t-1)}$, P-APG can compute $\pt^{(t)}$ within $K_t\triangleq\sigma_{\max}(C)\norm{\pt^*_{\gamma_t}-\pt^{(t-1)}}_2\sqrt{2/(\delta_t\gamma_t)}$ iterations. From Lemma~\ref{lem:dual-bound} and Remark~\ref{rem:dual_bounds}, it follows that $\norm{\pt^{(t-1)}}_1\leq B(\gamma_{t-1},\delta_{t-1},\alpha)\leq B_\theta(\gamma_{t-1},\alpha)+\frac{2\delta_{t-1}}{\alpha\upsilon}$, and $\norm{\pt^*_{\gamma_t}}_1\leq B_\theta(\gamma_{t},\alpha)$ for all $\alpha>0$. Therefore, Lemma~\ref{lem:opt-alpha} and $\gamma_{t-1}>\gamma_{t}$ imply
\begin{equation}
\norm{\pt^*_{\gamma_t}-\pt^{(t-1)}}_2\leq\norm{\pt^{(t-1)}}_1+\norm{\pt^*_{\gamma_t}}_1\leq 2B^*_\theta(\gamma_{t-1})+\frac{2\delta_{t-1}}{\alpha^*_{\gamma_{t-1}} \upsilon}.
\end{equation}
Hence, $K_t$, the number of P-APG iterations to compute $\pt^{(t)}$ can be bounded above as follows
\begin{equation}
\label{eq:K_t}
K_t\leq \bar{K}_t\triangleq\norm{C}\left(2B^*_\theta(\gamma_{t-1})+\frac{2\delta_{t-1}}{\alpha^*_{\gamma_{t-1}} \upsilon}\right)\sqrt{\frac{2}{\kappa_\gamma\kappa_\delta}}~\frac{1}{\epsilon_t}.
\end{equation}
From Lemma~\ref{lem:opt-alpha}, we have $\alpha^*_{\gamma_0}\leq\alpha^*_{\gamma_{t-1}}\leq\alpha^*_{\gamma_{t}}\leq \alpha^*$ for $t\geq 1$; hence, for $\gamma>0$,
\begin{equation*}
B^*_\theta(\gamma)
=B_\theta(\alpha^*_\gamma)+\frac{\gamma \alpha^*_\gamma}{\upsilon}\sum_{\ell\in\cN}\norm{\bar{x}_\ell-\hat{x}}_2^2
\leq B_\theta(\alpha^*_{\gamma_0})+\frac{\gamma \alpha^*}{\upsilon}\sum_{\ell\in\cN}\norm{\bar{x}_\ell-\hat{x}}_2^2.
\end{equation*}
Using this upper bound in \eqref{eq:K_t}, we can bound the total number of P-APG iterations needed to compute $\py_\epsilon$. In particular, $\py_\epsilon$ can be computed within
{\small
\begin{equation}
\label{eq:total-complexity}
\sum_{t=1}^{T_\epsilon}\bar{K}_t\leq\norm{C}\sqrt{\frac{2}{\kappa_\gamma\kappa_\delta}}\left[\frac{2\beta}{\upsilon}\left(\kappa_\gamma\alpha^*\sum_{\ell\in\cN}\norm{\bar{x}_\ell-\hat{x}}_2^2+\frac{\kappa_\delta}{\alpha^*_{\gamma_0}}\right)T_\epsilon+\frac{2B_\theta(\alpha^*_{\gamma_0})}{\epsilon_0}\sum_{t=1}^{T_\epsilon}\beta^t\right]
\end{equation}}%
P-APG iterations. Note that $\sum_{t=1}^{T_\epsilon}\beta^k=\frac{1}{\beta-1}(\beta^{T_\epsilon}-1)\leq \frac{\beta}{\beta-1}\frac{\Gamma^2}{\epsilon}$ since $T_\epsilon=\lceil\log_\beta(\Gamma^2/\epsilon)\rceil$. Therefore, $\sum_{t=1}^{T_\epsilon}\bar{K}_t=\cO(1/\epsilon)$.

To implement this scheme, for each outer iteration $t\geq 1$, \emph{verifiable} sufficient conditions for $\delta_t$-optimality can be used to terminate inner P-APG iterations. In fact, the number of P-APG iterations to compute $\pt^{(t)}$ is bounded above by $\bar{K}_t$, which can be computed a-priori; hence, giving us a stopping condition for the inner iterations. Moreover, one can also use other stopping conditions for inner iterations based on $\grad g_{\gamma_t}$ which are also sufficient for $\delta_t$ optimality; thus, making it possible to proceed to the next outer iteration before waiting for $\bar{K}_t$ inner iterations -- see Section~3.3 in~\cite{Aybat12} for a similar discussion. 
\vspace*{-2mm}
\subsection{Computational complexity of P-APG iterations}
\label{sec:complexity}
In Section~\ref{sec:PAPG}, we have seen that for any fixed $\epsilon$, setting $\gamma=\delta=\epsilon$ implies that $\norm{\pmb{y}_{\gamma,\delta}-\pmb{y}^*}_2^2=\cO(\epsilon)$ and $\pmb{y}_{\gamma,\delta}$ can be computed within $\cO(1/\epsilon)$ iterations of P-APG -- see Theorem~\ref{bound}. In Section~\ref{sec:cont-method}, we discussed that using continuation one can generate an iterate sequence $\{\py^t\}_{t\in\integers_+}$ such that $\py^{(t)}\rightarrow\py^*$ as $t\rightarrow +\infty$, and $\norm{\pmb{y}^{(t)}-\pmb{y}^*}_2^2=\cO(\epsilon)$ for all $t\geq T_\epsilon=\cO(1/\epsilon)$ for all $\epsilon>0$; moreover, computing $\py^{(T_\epsilon)}$ require at most $\cO(1/\epsilon)$ iterations of P-APG in total-- see \eqref{eq:total-complexity}.

The bottleneck operations at each P-APG iteration, displayed in Fig.~\ref{fig:papg}, are i) evaluating the matrix-vector multiplications with $C$ and $C^\top$, and ii) computing Step~1, which requires solving $K$ \emph{small}-size QPs. Matrix-vector multiplications with $C$ and $C^\top$ requires evaluating multiplications with $A_3$, $A_3^\top$, $A_4$ and $A_4^\top$. Moreover, since $A_3$ is a submatrix of $A_1 \in \mathbb{R}^{N(N-1) \times N}$, and $A_4$ is a submatrix of $A_2 \in \mathbb{R}^{N(N-1) \times Nn}$, as long as left and right vector multiplications with $A_1$ and $A_2$ can be done efficiently, one can do same operations with $C$ easily. Due to specific structures of $A_1$ and $A_2$, without forming $A_1$ and $A_2$ explicitly, one can compute $A_1 \py$ and ${A_1}^{\mathsf{T}} \pmb{z}$ with $\mathcal{O}(N^2 - N)$ complexity for all $\py$ and $\pmb{z}$; $A_2 \pxi$ and ${A_2}^{\mathsf{T}} \pmb{\omega}$ with $\mathcal{O} \big( n(N^2 - N) \big)$ complexity for all $\pxi$ and $\pmb{\omega}$. More importantly, neither $A_1$ nor $A_2$ is stored in the memory; storing only $\{\bar{x}_\ell\}_{\ell=1}^N$ is sufficient to be able to compute these matrix-vector multiplications.

First, we will consider the bottleneck step while solving \eqref{regularize} using a primal-dual 
IPM alone, without P-APG. This result will also help us understand the complexity of computing Step~1, which requires solving $K$ small size QPs as shown in~\eqref{subgradient_split}, which are in a similar form with the QP in \eqref{regularize}.

Let $c\in\reals^{N(n+1)}$ be an arbitrary vector,
$G=\left[
         \begin{array}{cc}
           \id_{N} & \zer \\
           \zer^\top & \gamma\id_{Nn} \\
         \end{array}
       \right]
$, and $A=[A_1~A_2]$ where $A_1\in\reals^{N(N-1)\times N}$ and $A_2\in\reals^{N(N-1)\times Nn}$ are defined in \eqref{eq:A}. Consider the generic QP
\begin{equation}
\label{eq:generic-QP}
\min_{\pe} \tfrac{1}{2}\pe^\T G \pe + c^\T \pe\quad\st\quad A\pe \geq 0:\pt,
\end{equation}
where $\pt\in\reals^{N(N-1)}$ is the vector of dual variables. Note that for appropriately chosen $c\in\reals^{N(n+1)}$, \eqref{regularize} is a special case of \eqref{eq:generic-QP}. Let $\pss\in\reals^{N(N-1)}$ represent the slack variables such that $\pss=[\pss_\ell]_{\ell\in\cN}$, where $\pss_\ell=[\pss_{\ell\ell'}]_{\ell'\in\cN\setminus\{\ell\}}\in\reals^{N-1}$. Given some $\tau>0$, the perturbed KKT system is given as
\begin{align}
\label{eq:KKT-perturbed}
&G\pe-A^\top\pt+c=0,\quad A\pe-\pss=0,\quad \pss_{\ell\ell'}\pt_{\ell\ell'}=\tau,\ (\ell,\ell')\in\cP,\\
&\pss\geq 0,\quad \pt\geq 0.\nonumber
\end{align}
Instead of directly solving the KKT system (for $\tau=0$), the primal-dual path following IPM methods inexactly solve the perturbed KKT conditions as $\tau\searrow 0$. Given $\tau>0$ and some point $(\pe,\pss,\pt)$ such that $\pss>0$ and $\pt>0$, the major operation is to compute the Newton direction for the nonlinear equation system in \eqref{eq:KKT-perturbed} from the given point. The Newton direction can be computed by solving the following system
\begin{align}
\label{eq:augmented_system}
\begin{bmatrix} G & - A^\T \\ A  & \Theta^{-1} \mathcal{S} \end{bmatrix} \begin{bmatrix} \quad {\Delta \pe \quad} \\ \Delta \pt \end{bmatrix} = \begin{bmatrix} -r_d \\ - r_p -\pss + \tau \Theta^{-1}\ones \end{bmatrix},
\end{align}
and setting $\Delta\pss=A\Delta\pe+r_p$, where $\mathcal{S} = \diag(\pss)$, $\Theta = \diag(\pt)$, $r_p = A\pe - \pss$, $r_d = G\pe - A^\T \pt + c$. \eqref{eq:augmented_system} implies that $\Delta\pe$ can be computed by solving
\begin{align}
\label{eq:normal-equation}
\Big(G + A^\T \mathcal{S}^{-1} \Theta A \Big) \Delta \pe = - r_d + A^\T \mathcal{S}^{-1} \Theta \Big(- r_p -\pss + \tau \Theta^{-1}\ones \Big).
\end{align}
It is easy to see that $M\triangleq G + A^\T \mathcal{S}^{-1} \Theta A$ is indeed a block arrowhead matrix. Indeed, let $\bd=\mathcal{S}^{-1} \Theta\ones\in\reals^{N(N-1)}$, i.e., $\bd_{\ell\ell'}=\pt_{\ell\ell'}/\pss_{\ell\ell'}$ for $(\ell,\ell')\in\cP$, and define $\bd_\ell=[\bd_{\ell\ell'}]_{\ell'\in\cN\setminus\{\ell\}}\in\reals^{N-1}$ for each $\ell\in\cN$. Since $A=[A_1 A_2]$, from the definition of $A_1$ and $A_2$ in \eqref{eq:A}, it follows that $M$ can be written as
{\small
\begin{equation}
\label{eq:arrowhead}
M = \begin{pmatrix}
	M_{00} & M_{01} & M_{02} &\cdots  & M_{0N} \\
	M_{01}^\top & M_{11} & \zer & \cdots & \zer\\
    M_{02}^\top & \zer^\top & M_{22} & \cdots & \zer\\
	\vdots & \vdots & \vdots & \ddots & \vdots \\
	M_{0N}^\top & \zer^\top & \zer^\top & \cdots & M_{NN} \\
	\end{pmatrix}
	\in \reals^{N(n+1) \times N(n+1)},\quad {\normalsize \hbox{where}}
\end{equation}}%
{\small
\begin{equation*}
M_{00}=\id_N+\sum_{\ell\in\cN}T_\ell^\top\diag(\bd_\ell)T_\ell,\quad M_{\ell\ell}=\gamma\id_n+X_\ell^\top\diag(\bd_\ell)X_\ell,\quad M_{0\ell}=T_\ell^\top\diag(\bd_\ell)X_\ell,
\end{equation*}
}%
for $\ell\in\cN$. Define $R_\ell\triangleq T_\ell^\top\diag(\bd_\ell)T_\ell$ for $\ell\in\cN$. Since $X_\ell=-T_\ell\bar{X}$, we have
\begin{equation}
\label{eq:M-components}
M_{00}=\id_N+\sum_{\ell\in\cN}R_\ell,\quad M_{\ell\ell}=\gamma\id_n+\bar{X}^\top R_\ell \bar{X},\quad M_{0\ell}=-R_\ell\bar{X},\quad \ell\in\cN.
\end{equation}
Moreover, due to structure of $T_\ell$ (see Definition~\ref{def:A1A2}), $R_\ell$ is a symmetric sparse matrix with a very special structure. In particular, it has only $3N-2$ nonzero elements, and $R_\ell X$ can be computed in $\cO(Nn)$ flops. Hence, forming $M_{\ell\ell}$ and $M_{0\ell}$ require $\cO(Nn(n+1))$ and $\cO(Nn)$ flops, respectively. It is easy to show that forming $M_{00}$ can be done in $\cO(N^2)$ flops; therefore, constructing $M$ requires $\cO(N^2n(n+2))$ flops in total.

In the next lemma,  we show that given an arbitrary $b\in\reals^{N(n+1)}$, the solution to the system $M\Delta\pe=b$ for $M$ given in~\eqref{eq:arrowhead} can be directly computed as in~\eqref{eq:Cholesky-closedform}. Alternatively, one can also compute the Cholesky factorization of $M$ first, and then use forward-backward substitution to compute the solution, which requires roughly the same amount of work that computing \eqref{eq:Cholesky-closedform} requires. In the proof of Theorem~\ref{lem:cholesky-block-arrowhead}, we also show as a side result that the Cholesky factorization of a generic block arrowhead matrix as in \eqref{eq:arrowhead} can be computed very efficiently, compared to factorization of a dense matrix.
\begin{theorem}
\label{lem:cholesky-block-arrowhead}
Let $M\in\reals^{N(n+1)\times N(n+1)}$ be a symmetric positive definite matrix with the generic block arrowhead structure given as in \eqref{eq:arrowhead}, where $M_{00}\in\reals^{N\times N}$, $M_{0\ell}\in\reals^{N\times n}$ and $M_{\ell\ell}\in\reals^{n\times n}$ for $\ell\in\cN$. Given arbitrary $b\in\reals^{N(n+1)}$ such that $b^\top=[b_0^\top b_1^\top \cdots b_N^\top]^\top$, the system $M\Delta\pe=b$ can be efficiently solved requiring $\cO(N^3+N^2n^2+2Nn^3)$ flops, where $b_0\in\reals^N$, $b_\ell\in\reals^n$ for $\ell\in\cN$, and $\Delta\pe^\top=[\Delta\py^\top \Delta\pxi_1^\top \cdots \Delta\pxi_N^\top]$. The solution is given as
{\small
\begin{equation}
\label{eq:Cholesky-closedform}
\Delta\py=M_{00}^{-1}\left(b_0-\sum_{\ell\in\cN}M_{0\ell}M_{\ell\ell}^{-1}b_\ell\right),\quad \Delta\pxi_\ell=M_{\ell\ell}^{-1}\left(b_\ell-M_{0\ell}^\top\Delta\py\right),\quad \ell\in\cN.
\end{equation}}%
\end{theorem}
\begin{proof}
In order to compute the Cholesky decomposition, we appropriately permute $M$ and consider the following equation system:
{\small
\begin{equation}
\label{eq:permuted}
\begin{pmatrix}
M_{11} & \zer & \cdots & \zer & M_{01}^\top\\
\zer^\top & M_{22} & \cdots & \zer & M_{02}^\top\\
\vdots & \vdots & \ddots & \vdots & \vdots \\
\zer^\top & \zer^\top & \cdots & M_{NN} & M_{0N}^\top\\
M_{01} & M_{02} &\cdots  & M_{0N} & M_{00}\\
\end{pmatrix}
\begin{pmatrix}
\Delta\pxi_1\\
\vdots\\
\Delta\pxi_N\\
\Delta\py
\end{pmatrix}
=
\begin{pmatrix}
b_1\\
\vdots\\
b_N\\
b_0
\end{pmatrix}
\end{equation}
}%
Let $M_{\rm per}$ be the matrix on the left hand side of \eqref{eq:permuted}. Compared to $M$, Cholesky decomposition of $M_{\rm per}$ can be computed much more efficiently. Indeed, diagonal blocks are factorized first: $M_{00}=F_0F_0^\top$, and $M_{\ell\ell}=F_\ell F_\ell^\top$ for $\ell\in\cN$. Since $M$ is positive definite, all the blocks on the diagonal are also positive definite; hence, $F_0$ and $F_\ell$ for $\ell\in\cN$ are invertible. The Cholesky factorization $M_{\rm per}=L_{\rm per}L_{\rm per}^\top$ can be easily verified:
{\small
\begin{equation}
\label{eq:Lperm}
L_{\rm per} =
\begin{pmatrix}
F_{1} & \zer & \cdots & \zer & \zer\\
\zer^\top & F_{2} & \cdots & \zer & \zer\\
\vdots & \vdots & \ddots & \vdots & \vdots \\
\zer^\top & \zer^\top & \cdots & F_{N} & \zer\\
L_1 & L_2 &\cdots  & L_N & F_{0}\\
\end{pmatrix},
\end{equation}}%
where $L_\ell=M_{0\ell}\left(F_\ell^{-1}\right)^\top$ for $\ell\in\cN$. Note that Cholesky factorization of $M_{\ell\ell}$ can be computed with $\cO(n^3)$ complexity for each $\ell\in\cN$, and with $\cO(N^3)$ for $\ell=0$. Hence, the total complexity for computing $\{F_\ell\}_{\ell\in\cN\cup\{0\}}$ is $\cO(N^3+Nn^3)$. Moreover, for each $\ell\in\cN$, computing $L_\ell$ requires $\cO(n^3)$ flops for inverting the lower diagonal matrix $F_\ell$, and $\cO(Nn^2)$ for the multiplication; thus, the total complexity of computing $L_{\rm per}$ is $\cO(N^3+N^2n^2+2Nn^3)$. Moreover, storing $L_{\rm per}$ requires roughly $N(n^2+N)/2+N^2n$ memory locations. Finally, computing $\Delta\pe$ requires one forward and one backward substitution which will roughly add another $\cO(Nn^2+N^2n)$ flops to the complexity.

Instead computing Cholesky factorization $M_{\rm per}=L_{\rm per}L_{\rm per}^\top$ explicitly, we will derive a closed form update rule for $\Delta\pe$. This will save us from storing $L_{\rm perm}$ and from doing additional forward-backward substitutions. First, we solve $L_{\rm per}\Delta r=b$ via forward substitution, where $\Delta r^\top=[\Delta r_1^\top \cdots \Delta r_N^\top \Delta r_o^\top]^\top$.
From \eqref{eq:Lperm}, it clearly follows that
\begin{equation}
\label{eq:forward}
\Delta r_0= F_0^{-1}(b_0-\sum_{\ell\in\cN}L_\ell\Delta r_\ell),\quad \Delta r_\ell=F_\ell^{-1}b_\ell,\quad \ell\in\cN.
\end{equation}
Next, we solve $L_{\rm per}^\top\Delta\pe=\Delta r$ for $\Delta\pe$ via backward substitution:
\begin{equation}
\label{eq:backward}
\Delta \py= \left(F_0^\top\right)^{-1}\Delta r_0,\quad \Delta \pxi_\ell=\left(F_\ell^\top\right)^{-1}\left(\Delta r_\ell-L_\ell^\top\Delta \py\right),\quad \ell\in\cN.
\end{equation}
Note that for each $\ell\in\cN$, from the definitions of $F_\ell$ and $L_\ell$, it follows that
\begin{equation}
\label{eq:fb-key}
L_\ell F_\ell^{-1}=M_{0\ell}\left(F_\ell^{-1}\right)^\top F_\ell^{-1}=M_{0\ell}M_{\ell\ell}^{-1}.
\end{equation}
Therefore, using \eqref{eq:forward}, \eqref{eq:backward}, and \eqref{eq:fb-key}, we can solve for $\Delta\pe$ in closed form as shown in \eqref{eq:Cholesky-closedform}.\qed
\end{proof}
As we discussed before, the bottleneck step while solving \eqref{regularize} using a primal-dual path following IPM is to solve either the augmented system in~\eqref{eq:augmented_system} or the normal equations in~\eqref{eq:normal-equation}. This reduces to computing the Cholesky decomposition of $M$ in~\eqref{eq:arrowhead} with components defined in~\eqref{eq:M-components} and using forward-backward substitution to compute $\Delta\pe$. Alternatively, according to Theorem~\ref{lem:cholesky-block-arrowhead}, one can also directly compute the solution as in~\eqref{eq:Cholesky-closedform}. Both alternatives have roughly the same complexity requiring $\cO(N^3+N^2n^2+2Nn^3)$ flops. Clearly, when $N$ is large, i.e., $N\geq 10^5$, this bottleneck step becomes impractical. On the other hand, combining P-APG and IPM, leaves the form of the bottleneck step unchanged, while making it more manageable by dividing it into smaller subsystem solves. In particular, the total complexity of computing Step~1 in P-APG consists of the complexity of solving $K$ small size QPs as shown in~\eqref{subgradient_split}.

Consider the problem in Step~1 of P-APG, and let $c=-C^\top\tilde{\pt}^k\in\reals^{N(n+1)}$. For each $i\in\cK$, define $c_i\in\reals^{\bar{N}(n+1)}$ such that $c_i$ is the subvector of $c$ corresponding to the indices of $\pe_i=[\py_i^\top \pxi_i^\top]^\top$, i.e., $\fprod{c,\pe}=\sum_{i\in\cK}\fprod{c_i,~\pe_i}$ for any $\pe$. Moreover, let
$\bar{G}=\left(
           \begin{array}{cc}
             \id_{\bar{N}} & \zer \\
             \zer & \id_{\bar{N}n} \\
           \end{array}
         \right)
$, and $\bar{A}^i=[\bar{A}^{ii}_1~\bar{A}^{ii}_2]$ where $\bar{A}^{ii}_1\in\reals^{\bar{N}(\bar{N}-1)\times \bar{N}}$ and $\bar{A}^{ii}_2\in\reals^{\bar{N}(\bar{N}-1)\times \bar{N}n}$ are defined in Definition~\ref{eq:submatrices}.
Hence, the problem in Step~1 of Fig.~\ref{fig:apg} can be equivalently written as
\begin{equation}
\label{eq:step1-problem}
\min_{\pe_i} \tfrac{1}{2}\pe_i^\T \bar{G} \pe_i + c_i^\T \pe_i\quad\st\quad \bar{A}^i\pe_i \geq 0:\pt_{ii},\qquad i\in\cK,
\end{equation}
where $\pt_{ii}\in\reals^{\bar{N}(\bar{N}-1)}$ is the vector of dual variables. For each $i\in\cK$, \eqref{eq:step1-problem} is in a similar form with the QP in \eqref{regularize}. Therefore, we immediately have the following result as a corollary of Theorem~\ref{lem:cholesky-block-arrowhead}.
\begin{corollary}
\label{cor:normal-equations}
For each $i\in\cK$, the normal equations corresponding to the QP in \eqref{eq:step1-problem} are in the same form with \eqref{eq:normal-equation} leading to a system with a block-arrowhead matrix as in \eqref{eq:arrowhead} with much smaller dimensions. Thus, Newton direction computations require $\cO(\bar{N}^3+\bar{N}^2n^2+2\bar{N}n^3)$ flops for each $i\in\cK$.
\end{corollary}
Suppose that we have $K$ parallel processors. It is worth noting that thanks to the separability of the problem in Step~1 of P-APG, i.e., \eqref{eq:step1-problem}, one can do this computation in parallel, running a primal-dual path following IPM on each one of the $K$ processors, or sequentially running the primal-dual path following IPM on a single processor $K$ times.
\begin{remark}
The total number of IPM iterations 
until P-APG terminates can be analyzed using the iteration complexity results on inexact accelerated proximal gradient algorithms~\cite{Schmidt11}, where Schmidt et al. analyzed APG in Fig.~\ref{fig:apg} when $\grad \rho$ in Step~1 is computed inexactly. In particular, one does not need to solve QP-subproblems exactly in each P-APG iteration. Given a tolerance sequence $\{\tau_k\}\subset\reals_{++}$ such that $\tau_k\searrow 0$, the number of primal path-following IPM iterations to compute a $\tau_k$-optimal solutions to QP-subproblems in the $k$-th iteration of P-APG is bounded above by $\cO(\bar{N}\ln(\frac{1}{\tau_k}))$ -- see Section~4.3.2 in~\cite{Nesterov04_1B} (similar bounds can be driven for primal-dual path-following IPMs as well). Moreover, since QP-subproblems are strongly convex, $\tau_k$-optimality in function values implies an error bound on gradient evaluations in Step~2 of P-APG.
\end{remark}

Recall that under Assumption~\ref{assump:partition}, we have $N=K\bar{N}$ such that $\bar{N}> n+1$. Below we consider the bottleneck memory requirement for solving \eqref{regularize} in 2 cases: running
 \textbf{a)} P-APG with a primal-dual IPM computing Step-1 in Fig.~\ref{fig:papg}, and \textbf{b)} IPM \emph{alone} on \eqref{regularize}. For case \textbf{a)}, the memory bottleneck in each iteration is due to solution of $K$ Newton systems corresponding to \eqref{eq:step1-problem}; on the other hand, for case~\textbf{b)}, the memory bottleneck is due to solution of a much larger Newton system using the normal equations in \eqref{eq:normal-equation}. In a naive implementation of case~\textbf{b)}, one stores the non-zero components of the Cholesky factor $L_{\rm per}$ in \eqref{eq:Lperm} corresponding to the block arrowhead matrix in~\eqref{eq:arrowhead} after permuting as in \eqref{eq:permuted}, which requires storing $\cO(N^2(n+1)+Nn^2)=K\cO(K\bar{N}^2(n+1)+\bar{N}n^2)$ entries; while for case~\textbf{a)}, for each $i\in\cK$, one stores the non-zero components of a Cholesky factor, analogous to \eqref{eq:Lperm}, for the QP in \eqref{eq:step1-problem} -- see Corollary~\ref{cor:normal-equations}; hence, this naive implementation requires storing $K\cO(\bar{N}^2(n+1)+\bar{N}n^2)$ entries for all the Cholesky factors in total, in addition to storing $\bar{N}^2 K(K-1)$ dual variables, i.e., $\pmb{\theta}=[\pt_{ij}]_{(i,j)\in\cG}\in\reals^{\bar{N}^2K(K-1)}$. Furthermore, for case~\textbf{b)}, in a more memory efficient implementation, \eqref{eq:Cholesky-closedform} in Theorem~\ref{lem:cholesky-block-arrowhead} implies that $\Delta\pxi_\ell$ can be computed sequentially after computing $\Delta\py$, which requires to store $\cO(N^2)$ at any time at the expense of forming $M_{\ell\ell}$ and $M_{0\ell}$ twice. Similarly, one can exploit this fact for case~\textbf{a)} as well while solving normal equations for each $i\in\cN$, which requires $\cO(K\bar{N}^2)$ memory in total if $K$ processors run in parallel, and $\cO(\bar{N}^2)$ if $K$ QPs in \eqref{eq:step1-problem} are solved sequentially on a single processor. Therefore, running IPM within P-APG reduces the memory requirement significantly at least by a factor of $K$ in comparison to running IPM alone, e.g., if we partition $N$ observations into $K=10$ subsets and each subproblem requires 1GB of memory, then running IPM alone requires roughly 100GB, while IPM within P-APG requires only 10GB in total. This discussion is summarized in Table~\ref{memory}. Finally, recall the discussion at the beginning of Section~\ref{sec:complexity}: neither $A_1$ nor $A_2$ needs to be stored in the memory; storing only $\{\bar{x}_\ell\}_{\ell=1}^N$ is sufficient to be able to compute matrix-vector multiplications with $A_1$ and $A_2$.
\vspace*{-2mm}

\begin{table}[htbp]
\centering
\caption{Comparison of Memory Usage}
{\scriptsize
\begin{tabular}{lll}
\toprule
 & \textbf{IPM alone} & \textbf{P-APG with IPM}\\
\toprule
Naive & $\cO\big( K^2 \bar{N}^2 n\big)$ & $\cO\big( K \bar{N}^2 (n+K)\big)$ \\
\midrule
Memory Efficient & $\cO\big( K^2 \bar{N}^2\big)$ & $\begin{array}{lc}
                                                     \cO\big( K \bar{N}^2\big) & \hbox{parallel}\\
                                                     \cO\big( \bar{N}^2\big) & \hbox{sequential}
                                                   \end{array}$
\\
\bottomrule
\end{tabular}}
\label{memory}
\end{table}
\section{Competitive Methods} In this section, we discuss an active set method for solving \eqref{regularize}, and a multi-block ADMM method recently proposed by~\cite{mazumder2015computational} to solve problem~\eqref{original_compact}.
\subsection{Active Set Method~(ASM)}
Although the number of constraints is $\cO(N^2)$ in \eqref{regularize}, one expects that only few of them will be potentially active at the optimal solution; furthermore, this indeed turned out to be the case based on our numerical results for the test problems we considered in this paper -- the number of active constraints was roughly $\cO(N)$. Therefore, in this section, we briefly state a primal active set method to solve the regularized convex regression problem in~\eqref{regularize} as an immediate alternative to P-APG method, and compare it with our P-APG method. One issue with primal active set methods is to determine an initial feasible point such that only very few constraints are active; and usually to overcome this problem one can use either ``Phase I" or ``big M" techniques. However, as we have already seen in the proof of Lemma~\ref{lem:dual-bound}, it is easy to construct an interior point for the polyhedron in \eqref{regularize} in spite of $\cO(N^2)$ constraints defining the set. In particular, let $\hat{x}\triangleq\frac{1}{N}\sum_{\ell\in\cN}\bar{x}_\ell$ and $\hat{y}\triangleq\frac{1}{N}\sum_{\ell\in\cN}\bar{y}_\ell$, and for an arbitrary $\alpha>0$, define $\tilde{\py}=[\tilde{y}_\ell]_{\ell\in\cN}$ and $\tilde{\pxi}=[\xi_\ell]_{\ell\in\cN}$ such that
\begin{equation}
\label{eq:slater-point-AS}
\tilde{y}_\ell\triangleq\hat{y}+\tfrac{\alpha}{2}\norm{\bar{x}_\ell-\hat{x}}_2^2,\quad \tilde{\xi}_\ell\triangleq\alpha(\bar{x}_\ell-\hat{x}),\quad \ell\in\cN.
\end{equation}
According to \eqref{eq:slater-point}, $\tilde{\pe}\triangleq[\tilde{\py}^\top \tilde{\pxi}^\top]^\top$ is a Slater point such that $A\tilde{\pe}=A_1\tilde{\py}+A_2\tilde{\pe}\geq \frac{\alpha\upsilon}{2}\ones$, where $A=[A_1 A_2]$ and $\upsilon>0$ is defined in Lemma~\ref{lem:dual-bound}. Hence, no constraint is active at $\tilde{\pe}$.

Consider \eqref{regularize}, which can be restated in a more compact form: the QP in~\eqref{eq:generic-QP} with $c=[\bar{\py}^\top \zer^\top]^\top$. We will show that the primal active set algorithm shown in Fig.~\ref{fig:asm} can be efficiently implemented. In the rest, $\ba_{\ell\ell'}^\top$ denotes the row of $A$ corresponding to index $(\ell,\ell')\in\cP$ -- recall that the rows of $A$ are sorted according to increasing lexicographic order on the index set $\cP$.
\begin{definition}
For $k\geq 1$, let $\cW^k\subset\cP$ denote the working set at iteration $k$, which is a subset of active constraint indices, i.e., $\ba_{\ell\ell'}^\top\pe^k=0$ for $(\ell,\ell')\in\cW^k$, and $m_k=|\cW^k|$. We form $A^k=[\ba_{\ell\ell'}^\top]_{(\ell,\ell')\in\cW^k}\in\reals^{m_k\times N(n+1)}$ concatenating the rows vertically, and define $A_1^k\in\reals^{m_k\times N}$ and $A_2^k\in\reals^{m_k\times Nn}$ as the submatrices of $A^k$ such that $A_1^k$ and $A_2^k$ consist of columns of $A^k$ corresponding to $\py$ and $\pxi$, respectively.
\end{definition}
The working set update strategy given in Fig.~\ref{fig:asm} ensures that $\{\ba_{\ell\ell'}\}_{(\ell,\ell')\in\cW^k}$ are \emph{linearly independent} for all $k\geq 1$ -- see~\cite{nocedal2006numerical} for details on this property, which we assume in the rest of this section.

\begin{figure}[t!]
\begin{framed}
{\small
\textbf{\underline{Algorithm ASM}}\\
Iteration 0: Set $\pe^0 \gets \tilde{\pe}$ as in \eqref{eq:slater-point-AS}, and $\cW^0 \gets \emptyset$\\
Iteration $k$: ($k \geq 0$)
	\begin{enumerate}[label=\arabic*:]
        	\item $\Delta\pe^k \gets \argmin\limits_{\Delta\pe} \tfrac{1}{2}\Delta\pe^\top G \Delta\pe + (G \pe^k +c)^{\T}\Delta\pe\quad\st\quad \ba_{\ell\ell'}^{\T}\Delta\pe = 0:\ \theta_{\ell\ell'},\quad (\ell,\ell')\in \cW^k$
        	\item \textbf{if} $\Delta\pe^k = 0$,  \textbf{then}
        	\item \qquad Compute $\{\theta_{\ell\ell'}\}_{(\ell,\ell')\in \cW^k}\subset\reals$ such that $ \sum_{(\ell,\ell') \in \cW^k} \ba_{\ell\ell'} \theta_{\ell\ell'}  = G\pe^k + c$
        	\item \qquad \textbf{if} $\theta_{\ell\ell'} \geq 0$ for all $(\ell,\ell') \in \cW^k$, \textbf{STOP} with solution $\pe^* = \pe^k$;
        	\item \qquad \textbf{else} $(\bar{\ell},\bar{\ell}') \gets \argmin_{(\ell,\ell') \in \cW^k} \theta_{\ell\ell'}$, $\pe^{k+1} \gets \pe^k$, $\cW^{k+1} \gets \cW^k \setminus \{ (\bar{\ell},\bar{\ell}') \} $;
        	\item \textbf{else} $\Delta\pe^k \neq 0$
        	\item \qquad $t^k \gets \min \left\{ 1,\ \min \middle\{\frac{-\ba_{\ell\ell'}^{\T} \pe^k}{\ba_{\ell\ell'}^{\T} \Delta\pe^k}:\ (\ell,\ell') \notin \cW^k\quad \st\quad \ba_{\ell\ell'}^{\T} \Delta\pe^k < 0 \middle\} \right\}$,
        	\item \qquad $\pe^{k+1} \gets \pe^k + t^k \Delta\pe^k$,\quad $\cI\gets\{(\ell,\ell')\notin \cW^k:\ \ba_{\ell\ell'}^\top \pe^{k+1}=0\}$
        	\item \qquad \textbf{if} $\cI\neq\emptyset$, \textbf{then} $\cW^{k+1}\gets W^k\cup \{(\ell,\ell')\}$ for some $(\ell,\ell')\in \cI$;
        	\item \qquad \textbf{else} set $\cW^{k+1} \gets \cW^k$;
        	\end{enumerate}
	\vspace{-0.25 cm}
}%
\end{framed}
\vspace{-0.25 cm}
\caption{Active Set Algorithm~(ASM)}
\label{fig:asm}
\vspace{-0.25 cm}
\end{figure}
Note that in each iteration $k\geq 1$, we need to solve a subproblem to determine the direction $\Delta\pe^k$ as follows
\begin{align}
\label{asm_sub}
\Delta\pe^k =\argmin\limits_{\Delta\pe} \tfrac{1}{2}\Delta\pe^\top G \Delta\pe + (G \pe^k +c)^{\T}\Delta\pe\quad\st\quad A^k\Delta\pe=0:\ \pt^k,
\end{align}
where $\pt^k=[\theta^k_{\ell\ell'}]_{(\ell,\ell')\in\cW^k}\in\reals^{m_k}$ denotes an optimal dual solution. Hence, $(\Delta\pe^k,\pt^k)$ satisfies the KKT system corresponding to \eqref{asm_sub}:
\begin{equation}
\label{eq:KKT-ASM}
\begin{bmatrix}
G & {A^k}^\top\\
A^k & \zer
\end{bmatrix}
\begin{bmatrix}
-\Delta\pe^k\\
\pt^k
\end{bmatrix}
=
\begin{bmatrix}
G\pe^k+c\\
0
\end{bmatrix}
\quad\Rightarrow\quad A^kG^{-1}{A^k}^\top\pt^k=A^k(\pe^k+c)
\end{equation}
since $G^{-1} c = c$. Therefore, $\pt^k$ can be computed via forward and backward substitution after computing the Cholesky factorization of $A^kG^{-1}{A^k}^\top$; next one can compute $\Delta\pe^k$ according to the first row in the KKT system as follows: $\Delta\pe^k=G^{-1}{A^k}^\top\pt^k-(\pe^k+c)$.
\begin{remark}
\label{rem:mult-bounds-ASM}
It is worth noting that $A^kG^{-1}{A^k}^\top=A_1^k{A_1^k}^\top+\frac{1}{\gamma}A_2^k{A_2^k}^\top$. For any $k\geq 1$, $A_1^k{A_1^k}^\top$ and $A_2^k{A_2^k}^\top$ computations require $\cO(m_k^2)$ and $\cO(m_k^2n)$ flops, respectively; and given $\ba_{\ell\ell'}$ for some $(\ell,\ell')\in\cP\setminus\cW^k$, $A^kG^{-1}\ba_{\ell\ell'}$ can be computed in $\cO(m_k(n+1))$ flops. These complexity bounds can be easily verified by observing the structure in $A^k$ after ordering its rows according to increasing lexicographic order on the index set $\cW^k$.
\end{remark}
Naively, the majority of total computational complexity at iteration $k$ is mainly due to forming $A^kG^{-1}{A^k}^\top=A_1^k{A_1^k}^\top+\frac{1}{\gamma}A_2^k{A_2^k}^\top$ in $\cO(m_k^2(n+1))$ flops, and computing its Cholesky factorization in $\cO(m_k^3)$ flops -- the factorization exists since $A^kG^{-1}{A^k}^\top$ is positive definite due to $\Rank(A^k)=m_k$. 
That said, at the end of each iteration the working set changes by at most one index; thus, one does not need to compute Cholesky factorizations from scratch. In particular, because at most one row (constraint) is added or deleted from $A^k$, Cholesky factorization for $A^{k+1}G^{-1}{A^{k+1}}^\top$ can be updated very efficiently by using 
$A^kG^{-1}{A^k}^\top=L^k{L^k}^\top$ from the previous iteration. Also, note $\pt^k$ is a byproduct of this approach, so we don't need to compute it again in the following step if $\pe^k=0$. Next, we will briefly discuss how to utilize the information from the previous iteration to solve the subproblems much more efficiently.
\begin{lemma}
\label{lem:cholesky-add-row}
For some $m\geq 1$, let $B\in\reals^{m\times N(n+1)}$ and $b\in\reals^{N(n+1)}$ such that $\Rank(B)=m$ and $b$ is not in the row-space of $B$. Suppose $LL^\top$ represent the Cholesky factorization of $B G^{-1} B^\top$ for some symmetric positive definite matrix $G$. Define
$\bar{B}=
\begin{bmatrix}
B\\
b^\top
\end{bmatrix}
\in\reals^{(m+1)\times N(n+1)}$. Then given $L$, Cholesky factorization for $\bar{B}G^{-1}\bar{B}^\top=\bar{L}\bar{L}^\top$ can be computed as
\begin{equation}
\label{eq:chol-update-add}
\bar{L}=
\begin{bmatrix}
L & \zer\\
h^\top & d
\end{bmatrix},
\quad h=L^{-1}BG^{-1}b,\quad d=\sqrt{b^\top G^{-1} b}.
\end{equation}
\end{lemma}
\begin{proof}
Since $\Rank(\bar{B})=m+1$, trivially $\bar{B}G^{-1}\bar{B}^\top$ is positive definite, and it has a Cholesky factorization $\bar{L}\bar{L}^\top$. Moreover, it is easy to verify that $\bar{L}$ given in \eqref{eq:chol-update-add} is the Cholesky factor.\qed
\end{proof}
Assume that we already know Cholesky factorization $A^kG^{-1}{A^k}^\top=L^k{L^k}^\top$, and $\cW^{k+1}=\cW^k\cup\{(\ell,\ell')\}$ for some $(\ell,\ell')\in\cP\setminus\cW^k$. Suppose $\ba_{\ell\ell'}^\top$ is appended to $A^k$ as the last row to form $A^{k+1}$. Since $\Rank(A^{k+1})=m_k+1$, setting $B=A^k$ and $b=\ba_{\ell\ell'}$ satisfies the conditions in Lemma~\ref{lem:cholesky-add-row}. Thus, according to \eqref{eq:chol-update-add}, the new factorization for $A^{k+1}G^{-1}{A^{k+1}}^\top=L^{k+1}{L^{k+1}}^\top$ can be computed as
$L^{k+1}=
\begin{bmatrix}
L^k & \zer\\
h^k & d^k
\end{bmatrix}
$, which only requires to solve $L^kh^k=A^kG^{-1}\ba_{\ell\ell'}$ for $h^k$, and to compute $d^k=\sqrt{\ba_{\ell\ell'}^\top G^{-1}\ba_{\ell\ell'}}$. Note computing $h^k$ requires forming $A^kG^{-1}\ba_{\ell\ell'}$, which can be computed in $\cO(m_k n)$ flops according to Remark~\ref{rem:mult-bounds-ASM}, and implementing one forward substitution, which can be done in $\cO(m_k^2)$ flops.

Now consider the case $\cW^{k+1}=\cW^k\setminus\{(\ell,\ell')\}$ for some $(\ell,\ell')\in\cW^k$. Note that $\ba_{\ell\ell'}^\top$ is an arbitrary row of $A^k$ (not necessarily the last one). The following lemma will help us update the factorization corresponding to $\cW^{k+1}$ efficiently when we are given $L^k$.
\begin{lemma}
Let $B_1\in\reals^{s_1\times N(n+1)}$, $B_2\in\reals^{s_2\times N(n+1)}$, and $b\in\reals^{N(n+1)}$ such that $\Rank(B)=s_1+s_2+1$, where
$B=
\begin{bmatrix}
B_1^\top & b & B_2^\top
\end{bmatrix}^\top$.
Suppose $LL^\top$ represent the Cholesky factorization of $B G^{-1} B^\top$ for some symmetric positive definite matrix $G$, where 
$L=
{\small
\begin{bmatrix}
L_1 & \zer_{s_1\times 1} & \zer_{s_1\times s_2}\\
h_1^\top & d & \zer_{1\times s_2}\\
F & h_2 & L_2
\end{bmatrix}}
$. Define
$\bar{B}=
\begin{bmatrix}
B_1^\top & B_2^\top
\end{bmatrix}^\top
\in\reals^{(s_1+s_2)\times N(n+1)}$. Then given $L$, Cholesky factorization for $\bar{B}G^{-1}\bar{B}^\top=\bar{L}\bar{L}^\top$ can be computed as
\begin{equation}
\label{eq:chol-update-minus}
\bar{L}=
\begin{bmatrix}
L_1 & \zer_{s_1\times s_2}\\
F & \bar{L}_2
\end{bmatrix},
\quad \st\quad \bar{L}_2\bar{L}_2^\top=L_2L_2^\top+h_2h_2^\top.
\end{equation}
Moreover, given $L_2$ and $h_2$, computing $\bar{L}_2$ requires $\cO(s_2^2)$ flops.
\end{lemma}
\begin{proof}
It is easy to verify that $\bar{L}$ given in \eqref{eq:chol-update-minus} is the lower-triangular Cholesky factor of $\bar{B}G^{-1}\bar{B}^\top$. For details of computing $\bar{L}_2$, refer to~\cite{gill1974methods}. Moreover, MATLAB routine \texttt{cholupdate}$(L_2^\top,h_2)$ can be called to compute $\bar{L}_2^\top$.\qed
\end{proof}
%
\begin{figure}[h!]
\begin{framed}
{\small
\textbf{\underline{Algorithm Multi-Block ADMM}}\\
Iteration 0: $\Delta_{ij}\gets\bar{x}_i-\bar{x}_j$ for $(i,j)\in\cN\times\cN$ and $\bar{\Delta}_j \gets ( \sum_{i\in\cN} \Delta_{ij}\Delta_{ij}^\T)^{-1}$ for $j\in\cN$\\
Iteration $k$: ($k \geq 0$)
	\begin{enumerate}[label=\arabic*:]
        	\item $\xi_j^{k+1} \gets \bar{\Delta}_j \big( \sum_{i\in\cN} \Delta_{ij} (\theta_{ij}^k/\rho + \nu_{ij}^k + y_i^k - y_j^k ) \big)$ for $j\in\cN$
            \item $\tilde{\nu}^{k+1}_{ij}\gets\nu_{ij}^k - \Delta_{ij}^\T \xi_j^{k+1}$ for $(i,j)\in\cN\times\cN$
        	\item $\pmb{w}^{k+1} \gets \bar{\pmb{y}} + D^\T \pt^{k} + \rho D^\T \tilde{\pmb{\nu}}^{k+1}$
            \item $y_i^{k+1} \gets \frac{1}{1+2N\rho} \left( w_i + 2\rho \sum_{j\in\cN} w_j \right)$ for $i\in\cN$
            \item $\nu_{ij}^{k+1} \gets \min \left\{ y_j^{k+1} + \Delta_{ij}^\T \xi_j^{k+1} - y_i^{k+1} -  \theta_{ij}^k/\rho, 0 \right\}$ for $(i,j)\in\cN\times\cN$
            \item $\theta_{ij}^{k+1} \gets \theta_{ij}^{k} + \rho \left( \nu_{ij}^{k+1} + y_i^{k+1}- y_j^{k+1}-\Delta_{ij}^\T \xi_j^{k+1} \right)$ for $(i,j)\in\cN\times\cN$
    \end{enumerate}
	\vspace{-0.25 cm}
}%
\end{framed}
\vspace{-0.25 cm}
\caption{Multi-block ADMM~(ADMM)}
\label{fig:admm}
\vspace*{-5mm}
\end{figure}
\subsection{Multi-block ADMM}
Recently, Mazumder et al.~\cite{mazumder2015computational} proposed a multi-block ADMM to solve problem~\eqref{original_compact}. Although, the authors report that it works well in practice, to our best knowledge, the convergence property of the method is still unknown. In fact, it is recently shown that ADMM does not necessarily converge when the number of primal variable blocks are three or more~\cite{chen2016direct}; and the ADMM algorithm in \cite{mazumder2015computational}, displayed in Fig.~\ref{fig:admm}, alternatingly updates \emph{three}-blocks of primal variables: $\pmb{\xi}=[\xi_i]_{i\in\cN}$, $\py=[y_i]_{i\in\cN}$ and $\pmb{\nu}=[\nu_{ij}]_{(i,j)\in\cN\times\cN}$.

The matrix $D \in \reals^{N^2 \times N}$ is similar to our matrix $A_1$ defined in Definition~\ref{def:A1A2}, except $D$ also contains rows corresponding to $(i,i)\in\cN\times\cN$, i.e., $D\py=\pmb{z}\in\reals^{N^2}$ such that $z_{ij}=y_j-y_i$ for $(i,j)\in\cN\times\cN$, and the long-vector $\pmb{z}$ obtained by sorting its elements according to increasing lexicographic order on the index set $\cN\times\cN$. Similarly, the elements of the auxiliary variable $\tilde{\pmb{\nu}}\in\reals^{N^2}$ is also sorted according to increasing lexicographic order on the index set $\cN\times\cN$. During initialization, the ADMM algorithm requires computing $\bar{\Delta}_j = ( \sum_{i\in\cN} \Delta_{ij}\Delta_{ij}^\T)^{-1}$ for all $i\in\cN$, where $\Delta_{ij} = \bar{x}_i - \bar{x}_j$. Although it is required only one time, this computation costs $\cO(N^2n^2 + Nn^3)$ flops. Based on our numerical tests, as $N$ increases, this preprocessing time becomes substantial compared to overall runtime. At each iteration, the algorithm needs to update five different variables: $\pxi^k$, $\pmb{w}^k$, $\py^k$, $\pmb{\nu}^k$ and $\pt^k$. The cost for updating subgradient vector $\pxi^k$ is $\cO(N^2n+Nn^2)$ flops, updating $\pmb{w}^k$ takes $\cO(N^2)$ flops, and given $\pmb{w}^{k-1}$ updating the function value-vector $\py^k$ takes $\cO(N)$ flops, and updating residuals $\pmb{\nu}^k$ and dual variables $\pt^k$ both take $\cO(N^2)$ flops separately. Thus, the overall per iteration complexity is $\cO(N^2n+Nn^2)$ with $\cO(N^2n^2 + Nn^3)$ one-time cost at the beginning.  Note that ADMM needs to store not only matrix $D \in \reals^{ N^2 \times N }$ and vectors $\Delta_{ij} \in \reals^{n}$ for all $(i,j)\in\cN$, that are comparable to our $A_1$ and $A_2$, but also $\bar{\Delta}_i \in \reals^{n \times n}$ for all $\in\cN$, which are the matrices inverted during pre-processing; hence, the number of non-zeros stored in the RAM for ADMM is roughly $(N^2 - N)(n+2)+Nn^2$, which is $\cO(K^2\bar{N}^2n)$. When compared to Table~\ref{memory}, clearly P-APG leads to significant memory savings.

\vspace*{-4mm}
\section{Numerical Study}
\label{numerical} \vspace*{-3mm}

Here we demonstrate the scalability of P-APG, and compare its performance against other competitive methods: an interior point method, an active set method~(ASM), and a multi-block ADMM. To solve the convex regression problem, we implemented P-APG, ASM and ADMM in MATLAB, and used the stand-alone version MOSEK~\cite{mosek} as an interior point solver for benchmarking purposes. Moreover, for P-APG, we also use MOSEK together with the Parallel Computing Toolbox, in order to solve K QP-subproblems in \emph{parallel} using $K$ cores in each iteration of P-APG in Fig.~\ref{fig:papg}. 
MOSEK is a commercial off-the-shelf software which has a state-of-the-art interior-point optimizer for quadratic problems. 
Note that MOSEK also comes with CVX, which is a popular MATLAB-based modeling system for convex optimization; but this version of MOSEK is not compatible with Parallel Computing Toolbox in MATLAB, i.e., even though one calls MOSEK through CVX formulations within a \texttt{parfor} loop, the $K$ subproblems are still solved in a sequential manner. In order to take advantage of the computing power in a cluster of computers for long-running jobs, one has to adopt \emph{batch} processing in MATLAB to be able to better exploit the processor cores in multiple machines. On the other hand, matrix operations in MATLAB leverage multi-core and multi-threading framework by default. Hence, ASM and ADMM are coded without using the parallel toolbox, as they only contain matrix operations in every iteration and these operations are executed in parallel automatically. To eliminate factors that might have an influence on the runtime to the best extent, we carried out all numerical tests comparing P-APG against 
other methods on high performance computing cluster by executing a single script, so that they all run on exactly the same processor cores and memory modules. Numerical tests 
are carried out on a single node at a research computing cluster. The node is composed of one 24-core processor, each having 1GB RAM (24GB RAM in total). We determine the number of core processors and the amount of RAM allocated depending on the size of the problem solved -- 
see Sections~\ref{sec:self-test} and~\ref{sec:tests}.\vspace*{-5mm} 
\paragraph{Experimental setup:}
Our problem setup adopted in the following sections involve two different test functions: 1) $f_0(\pmb{x})=\frac{1}{2}\pmb{x}^{\mathsf{T}} Q \pmb{x}$ and 2) $f_0(\pmb{x}) = \exp( \pmb{p}^{\mathsf{T}} \pmb{x} )$, where  $Q\in\reals^{n\times n}$ is a symmetric matrix, $\pmb{p} \in \mathbb{R}^n$, and they are randomly generated as follows. We first set $\bar{Q}\triangleq\Lambda^{\mathsf{T}} \Lambda$ such that $\Lambda\in\mathbb{R}^{n\times n}$ is generated randomly with all components being i.i.d. with $\cN(0,1)$, where $\cN(\mu,\sigma^2)$ denotes Normal distribution with mean $\mu$ and variance $\sigma^2$; next, without changing left and right singular vectors of $\bar{Q}$, we transform its singular values such that the resulting condition number is 15 and we call the resulting matrix as $Q$; and $\pmb{p} \in \mathbb{R}^n$ is generated using uniform distribution on the hypercube $[0,0.2]^n$. The noisy observations $\{\bar{y}_\ell\}_{\ell\in\cN}$ are generated according to \eqref{data}, where the locations $\{ \bar{x}_\ell \}_{\ell=1}^N\subset\mathbb{R}^n$ and additive noise $\{ \epsilon_\ell \} _{\ell=1}^N\subset\mathbb{R}$ are generated randomly with all components being i.i.d. with $\cN(0,4)$ and $\cN(0,100)$ respectively. In addition, we moved 30\% of randomly chosen location/observation pairs into the interior of the epigraph of the test function $f_0$ by replacing $(\bar{x}_\ell,\bar{y}_\ell)$ with $(\bar{x}_\ell, 1.3\bar{y}_\ell)$. In all the experiments involving P-APG and ASM, we set $\gamma=10^{-4}$ in \eqref{regularize}.
\vspace*{-3mm}
\subsection{Convergence behavior of P-APG on the regularized problem}
\label{sec:self-test}
We compare i) running MOSEK alone and ii) running it within P-APG on the regularized problem~\eqref{regularize} with increasing dimension. The numerical study is mainly aimed to demonstrate how the performance of each method scales for solving the \emph{regularized} problem as its dimension increases.
First, we start with a small size problem: $n=10$, $N=100$, and use the test function $f_0(\pmb{x})=\frac{1}{2}\pmb{x}^{\mathsf{T}} Q \pmb{x}$.
We 
compare the quality of the solutions computed by P-APG and dual gradient ascent (as the dual function $g_\gamma$ in \eqref{eq:g_gamma} is differentiable). In order to compute dual gradient, $\grad g_\gamma$, one needs to solve $K$ quadratic subproblems. To exploit this parallel structure, we partition the data into two sets, i.e., $K=2$. Within both dual gradient ascent and P-APG, we called MOSEK to compute the dual gradients via solving $K$ QP-subproblems. Since we allow violations for the relaxed constraints, we define the ``duality gap" at the $k$-th iteration as $\pmb{\theta}_k^{\mathsf{T}} C \pmb{\eta}_k$ -- recall that $\pmb{\eta}_k=[\pmb{y}_k^\top \pmb{\xi}_k^\top]^\top$. Fig.~\ref{fig:gap}(left) represents how the duality gap for both methods change at each iteration. In order to better understand the behavior of P-APG, we report in Fig.~\ref{fig:gap}(right) the duality gap of P-APG in a larger scale. Fig.~\ref{Distance_All} reports the infeasibility of iterates, i.e., $\big\| \big( A_1 \, \pmb{y}_k +A_2 \, \pmb{\xi}_k \big)_{-} \big\|_2$. \vspace*{-4mm}

\begin{figure}[h!]
\centering
\includegraphics[width=0.48\textwidth]{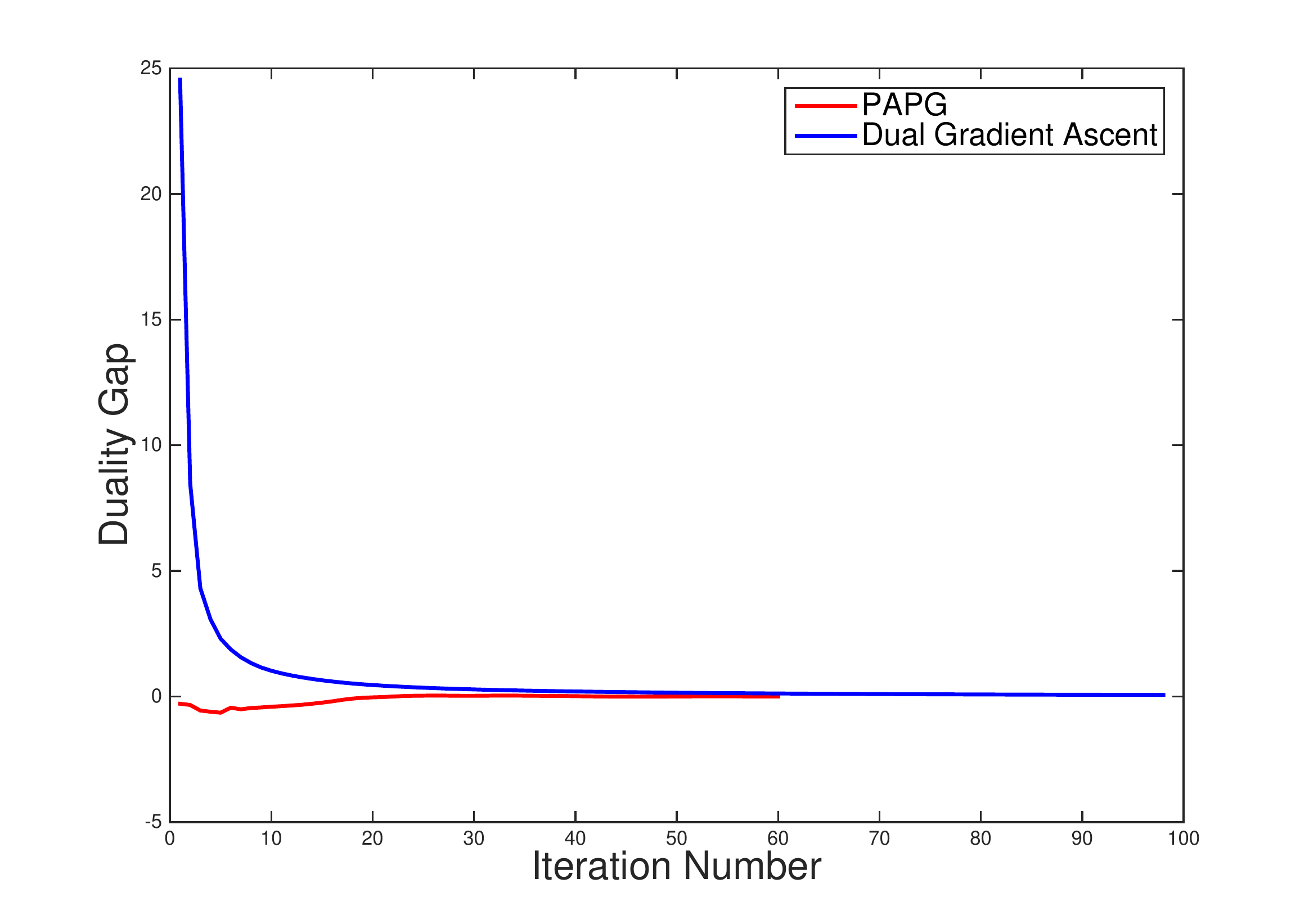}
\includegraphics[width=0.48\textwidth]{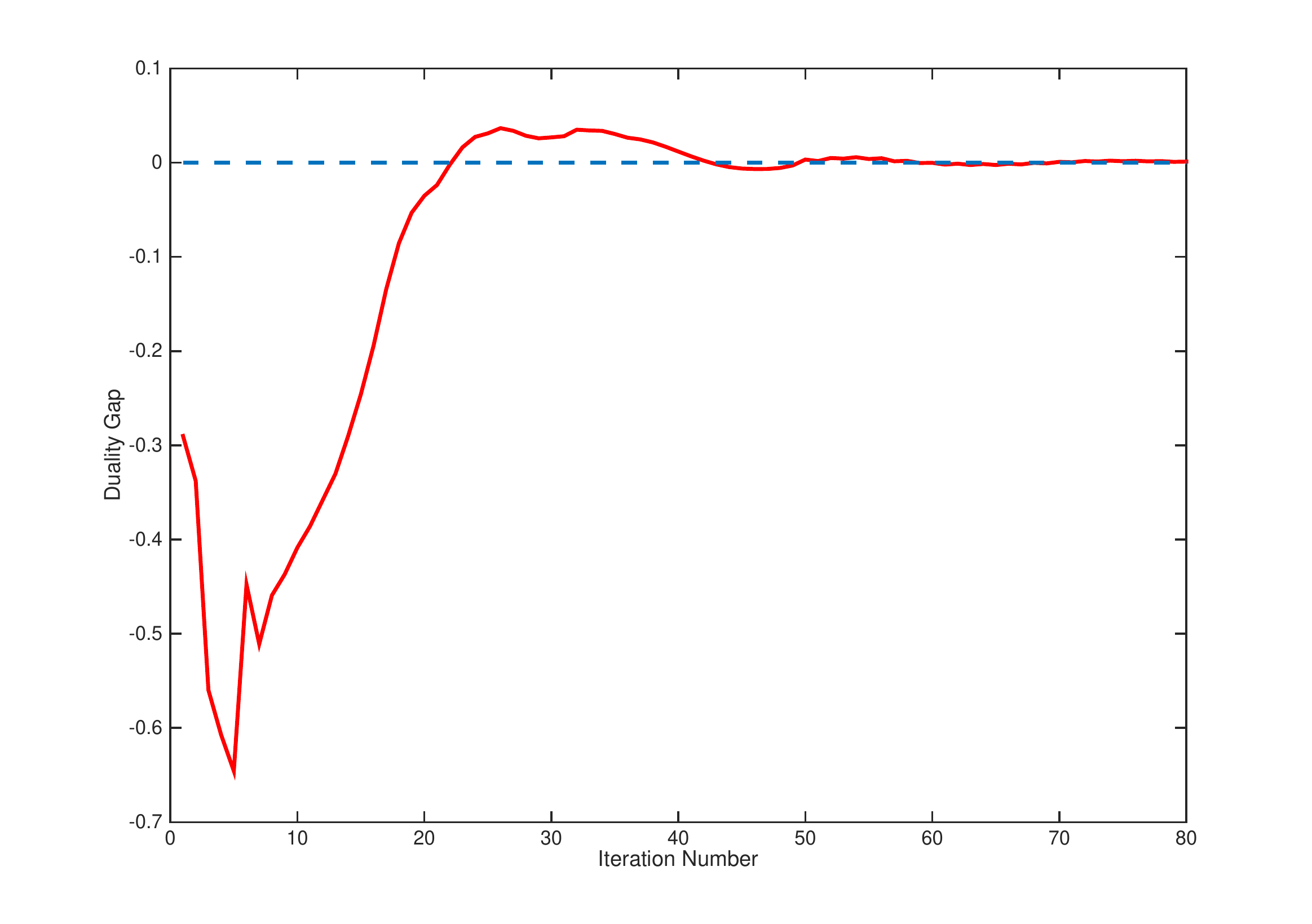}
\caption{ Duality Gap for P-APG and Dual Gradient Ascent: (left) P-APG and Dual Gradient Ascent, (right) Zoom-in for P-APG Method}
\label{fig:gap}
\vspace*{-2mm}
\end{figure}

\begin{figure}[h!]
\centering
\includegraphics[width=0.5\textwidth]{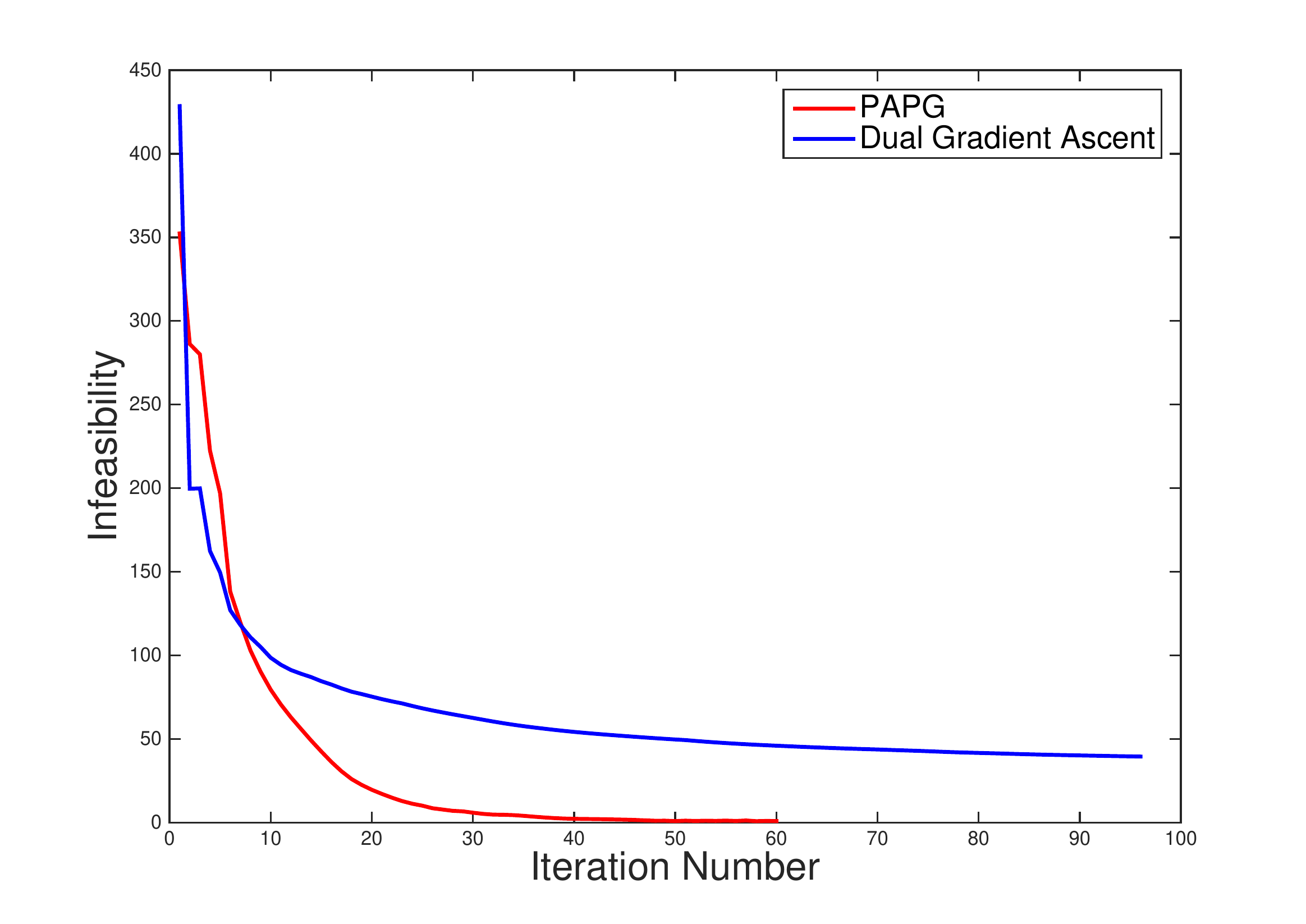}
\caption{ Distance to Feasible Region for P-APG and Dual Gradient Ascent. 
}
\label{Distance_All}
\vspace*{-5mm}
\end{figure}

\begin{table}[b]
\caption{ Comparison with test function $\frac{1}{2}\pmb{x}^{\mathsf{T}}Q\pmb{x}$ for $n=20$}
\label{tab:xqx_20}
\centering
{\tiny
\begin{tabular}{llllllll}
\toprule
n, N & Algorithm  & Cores/RAM & Preprocess & Wall-time & Infeasibility & SubOpt\_Reg & DualGap \\
\midrule
\multirow{3}{*}{20, 200} & Mosek\_Reg & 2/2 & 0 & 2 & 0 & 0 & 0 \\
 & PAPG\_A & 2/2 & 0.2 & 30 & 9.51E-02 & 3.86E-04 & 6.80E-08 \\
 & PAPG\_C & 2/2 & 0.2 & 16 & 9.75E-02 & 1.22E-03 & 2.63E-07 \\
\midrule
\multirow{3}{*}{20, 400} & Mosek\_Reg & 4/4 & 0 & 9 & 0 & 0 & 0 \\
 & PAPG\_A & 4/4 & 1.1 & 47 & 9.54E-02 & 1.92E-03 & 2.56E-07 \\
 & PAPG\_C & 4/4 & 1.0 & 44 & 9.94E-02 & 7.45E-05 & 3.23E-08 \\
\midrule
\multirow{3}{*}{20, 800} & Mosek\_Reg & 8/8 & 0 & 42 & 0 & 0 & 0 \\
 & PAPG\_A & 8/8 & 4.3 & 92 & 9.92E-02 & 6.14E-03 & 4.79E-07 \\
 & PAPG\_C & 8/8 & 4.2 & 101 & 9.87E-02 & 2.06E-03 & 1.75E-07 \\
\midrule
\multirow{3}{*}{20, 1600} & Mosek\_Reg & 16/16 & 0 & 311 & 0 & 0 & 0 \\
 & PAPG\_A & 16/16 & 20.3 & 259 & 9.50E-02 & 1.10E-02 & 4.80E-07 \\
 & PAPG\_C & 16/16 & 20.0 & 408 & 9.99E-02 & 5.89E-03 & 2.63E-07 \\
\midrule
\multirow{3}{*}{20, 2400} & Mosek\_Reg & 24/24 & 0 & 987 & 0 & 0 & 0 \\
 & PAPG\_A & 24/24 & 56.1 & 323 & 9.77E-02 & 1.43E-02 & 4.61E-07 \\
 & PAPG\_C & 24/24 & 59.7 & 723 & 1.00E-01 & 7.42E-03 & 2.43E-07 \\
\bottomrule
\end{tabular}
}
\vspace*{-3mm}
\end{table}%
\begin{table}[h]
\caption{Comparison with test function $\frac{1}{2}\pmb{x}^{\mathsf{T}}Q\pmb{x}$ for $n=80$}
\label{tab:xqx_80}
\centering
{\tiny
\begin{tabular}{llllllll}
\toprule
n, N & Algorithm  & Cores/RAM & Preprocess & Wall-time & Infeasibility & SubOpt\_Reg & DualGap \\
\midrule
\multirow{3}{*}{80, 200} & Mosek\_Reg & 2/2 & 0 & 4 & 0 & 0 & 0 \\
 & PAPG\_A & 2/2 & 0.5 & 89 & 9.46E-02 & 6.08E-04 & 1.88E-07 \\
 & PAPG\_C & 2/2 & 0.5 & 54 & 9.93E-02 & 1.05E-03 & 3.28E-07 \\
\midrule
\multirow{3}{*}{80, 400} & Mosek\_Reg & 4/4 & 0 & 18 & 0 & 0 & 0 \\
 & PAPG\_A & 4/4 & 2.5 & 324 & 8.92E-02 & 5.79E-04 & 9.54E-08 \\
 & PAPG\_C & 4/4 & 2.6 & 287 & 9.87E-02 & 1.09E-03 & 1.96E-07 \\
\midrule
\multirow{3}{*}{80, 800} & Mosek\_Reg & 8/8 & 0 & 97 & 0 & 0 & 0 \\
 & PAPG\_A & 8/8 & 12.2 & 383 & 9.79E-02 & 1.43E-03 & 1.29E-07 \\
 & PAPG\_C & 8/8 & 12.5 & 379 & 9.92E-02 & 1.02E-03 & 9.12E-08 \\
\midrule
\multirow{3}{*}{80, 1600} & Mosek\_Reg & 16/16 & 0 & 661 & 0 & 0 & 0 \\
 & PAPG\_A & 16/16 & 55.9 & 597 & 9.59E-02 & 1.61E-03 & 7.67E-08 \\
 & PAPG\_C & 16/16 & 57.5 & 1145 & 9.94E-02 & 1.61E-03 & 7.70E-08 \\
\midrule
\multirow{3}{*}{80, 2400} & Mosek\_Reg & 24/24 & 0 & 1966 & 0 & 0 & 0 \\
 & PAPG\_A & 24/24 & 128.5 & 897 & 9.91E-02 & 1.17E-03 &  3.87E-08 \\
 & PAPG\_C & 24/24 & 133.8 & 1947 & 1.00E-01 & 2.57E-03 & 9.19E-08 \\
\bottomrule
\end{tabular}
}
\vspace*{-3mm}
\end{table}%
\begin{table}[t]
\caption{ Comparison with test function $\exp( \pmb{p}^{\mathsf{T}}\pmb{x})$ for $n=20$}
\label{tab:exp_20}
\centering
{\tiny
\begin{tabular}{llllllll}
\toprule
n, N & Algorithm  & Cores/RAM & Preprocess & Wall-time & Infeasibility & SubOpt\_Reg & DualGap \\
\midrule
\multirow{3}{*}{20, 200} & Mosek\_Reg & 2/2 & 0 & 3 & 0 & 0 & 0 \\
 & PAPG\_A & 2/2 & 0.2 & 35 & 8.97E-02 & 1.01E-03 & 3.58E-07 \\
 & PAPG\_C & 2/2 & 0.2 & 17 & 9.62E-02 & 1.16E-03 & 4.09E-07 \\
\midrule
\multirow{3}{*}{20, 400} & Mosek\_Reg & 4/4 & 0 & 10 & 0 & NaN & 0 \\
 & PAPG\_A & 4/4 & 1.0 & 67 & 9.41E-02 & 1.36E-03 & 2.15E-07 \\
 & PAPG\_C & 4/4 & 1.0 & 71 & 9.93E-02 & 6.50E-04 & 7.11E-08 \\
\midrule
\multirow{3}{*}{20, 800} & Mosek\_Reg & 8/8 & 0 & 66 & 0 & 0 & 0 \\
 & PAPG\_A & 8/8 & 4.2 & 194 & 7.04E-02 & 5.56E-03 & 4.62E-07 \\
 & PAPG\_C & 8/8 & 4.1 & 266 & 9.91E-02 & 5.39E-03 & 4.45E-07 \\
\midrule
\multirow{3}{*}{20, 1600} & Mosek\_Reg & 16/16 & 0 & 558 & 0 & 0 & 0 \\
 & PAPG\_A & 16/16 & 18.7 & 553 & 9.75E-02 & 6.14E-02 & 3.66E-07 \\
 & PAPG\_C & 16/16 & 18.8 & 886 & 9.72E-02 & 1.65E-03 & 4.51E-07 \\
\midrule
\multirow{3}{*}{20, 2400} & Mosek\_Reg & 24/24 & 0 & 2155 & 0 & 0 & 0 \\
 & PAPG\_A & 24/24 & 128.5 & 797 & 9.97E-02 & 5.79E-02 & 1.29E-07 \\
 & PAPG\_C & 24/24 & 133.8 & 1347 & 1.00E-01 & 2.92E-03 & 3.59E-08 \\
\bottomrule
\end{tabular}
}
\vspace*{-3mm}
\end{table}%
\begin{table}[h!]
\caption{ Comparison with test function $\exp( \pmb{p}^{\mathsf{T}}\pmb{x})$ for $n=80$}
\label{tab:exp_80}
\centering
{\tiny
\begin{tabular}{llllllll}
\toprule
n, N & Algorithm  & Cores/RAM & Preprocess & Walltime & Infeasibility & SubOpt\_Reg & DualGap \\
\midrule
\multirow{3}{*}{80, 200} & Mosek\_Reg & 2/2 & 0 & 4 & 0 & 0 & 0 \\
 & PAPG\_A & 2/2 & 0.4 & 53 & 9.21E-02 & 6.68E-03 & 4.25E-07 \\
 & PAPG\_C & 2/2 & 0.4 & 34 & 9.27E-02 & 1.60E-03 & 1.01E-07 \\
\midrule
\multirow{3}{*}{80, 400} & Mosek\_Reg & 4/4 & 0 & 21 & 0 & 0 & 0 \\
 & PAPG\_A & 4/4 & 2.6 & 209 & 9.39E-02 & 3.00E-03 & 1.15E-07 \\
 & PAPG\_C & 4/4 & 2.6 & 183 & 9.63E-02 & 2.92E-03 & 1.12E-07 \\
\midrule
\multirow{3}{*}{80, 800} & Mosek\_Reg & 8/8 & 0 & 143 & 0 & 0 & 0 \\
 & PAPG\_A & 8/8 & 12.8 & 299 & 9.12E-02 & 9.44E-04 & 1.90E-08 \\
 & PAPG\_C & 8/8 & 13.0 & 272 & 9.87E-02 & 1.65E-04 & 5.01E-10 \\
\midrule
\multirow{3}{*}{80, 1600} & Mosek\_Reg & 16/16 & 0 & 687 & 0 & 0 & 0 \\
 & PAPG\_A & 16/16 & 55.3 & 465 & 9.72E-02 & 3.07E-03 & 3.59E-08 \\
 & PAPG\_C & 16/16 & 54.7 & 753 & 9.95E-02 & 1.28E-04 & 1.53E-09 \\
\midrule
\multirow{3}{*}{80, 2400} & Mosek\_Reg & 24/24 & 0 & 2789 & 0 & 0 & 0 \\
 & PAPG\_A & 24/24 & 138.2 & 774 & 9.69E-02 & 7.82E-03 & 6.39E-08 \\
 & PAPG\_C & 24/24 & 139.3 & 1379 & 9.94E-02 & 2.23E-03 & 2.01E-08 \\
\bottomrule
\end{tabular}
}
\end{table}

A primal-dual iterate $( \pmb{\eta} ,\pmb{\theta})$ is optimal if the duality gap and infeasibility are both zero. As the feasibility happens in the limit, the duality gap in Fig.~\ref{fig:gap}(right) can go below 0, which can be explained by the infeasibility of iterates. Therefore, observing a decrease in duality gap only tells one part of the story; without convergence to feasibility, it is not valuable alone as a measure. As shown in the Fig.~\ref{fig:gap}, the duality gap converges quickly to zero for both methods. On the other hand, as shown in Fig.~\ref{Distance_All}, constraint violation for P-APG iterates decreases to 0 much faster than it does for the dual gradient ascent iterates. Hence, P-APG iterate sequence converges to the unique optimal solution considerably faster.

As shown in Tables~\ref{tab:xqx_20},~\ref{tab:xqx_80},~\ref{tab:exp_20} and~\ref{tab:exp_80}, the dimension of variables $n\in\{20,80\}$, and the number of observations $ N\in\{200, 400, 800, 1600, 2400\}$. Since the number of constraints increases at the rate of $\mathcal{O}(N^2)$, as the size of problem increases in $N$, we reported the normalized infeasibility $\norm{\big( A_1\pmb{y} +A_2\pmb{\xi} \big)_{-}}_2/\sqrt{N^2-N}$ and normalized duality gap $|\pmb{\theta}^{\mathsf{T}} C \pmb{\eta}|/(N^2-N)$. We partition the set of observations $\cN$ into $K$ subsets. Each one of them consists of 100 points; therefore, we set $K=2, 4, 8, 16, 24$ for $N=200, 400, 800, 1600, 2400$, and we reserve $2/2, 4/4, 8/8, 16/16, 24/24$ number of \texttt{Cores/RAM}, respectively, depending on $N$ so that for each job submitted to the computing cluster, an instance of \eqref{regularize} is solved using P-APG on the node such that each subproblem in \eqref{eq:step1-problem} for $i\in\cK$ is computed on a different core. We tested both the adaptive step and constant step version of P-APG, which we abbreviate as PAPG\_A and PAPG\_C, respectively. Both PAPG\_A and PAPG\_C are terminated whenever they compute a primal-dual iterate, $\pmb{\eta}=[\py^\top \pxi^\top]^\top$ and $\pt$, satisfying the stopping criteria: $\norm{ (A_1 \py + A_2 \pxi)_{-}}/\sqrt{N^2-N)} \leq$ 1e-1 and $ | \pmb{\theta}^{\mathsf{T}} C \pmb{\eta} | / (N^2-N)  \leq $ 5e-7, or at the end of 2 hours, which are reported as \texttt{Infeasibility} and \texttt{DualGap} respectively in the tables. Moreover, we also report relative suboptimality, i.e., \texttt{SubOpt\_Reg}$=| p - p_\gamma^* | / p_\gamma^*$, where $p_\gamma^*$ denotes the optimal value to~\eqref{regularize} and $p$ denotes the objective value of~\eqref{regularize} at termination. \texttt{Preprocess} for P-APG method is the wall-clock time elapsed during the computation of the maximum singular value for the matrix $A_4$. Additionally, in all the tables, \texttt{N/A} means that the wall clock time exceeded 2 hours for the job, 
and \texttt{Wall-time} stands for wall-clock time in \emph{seconds} for the whole job including \texttt{Preprocess}.

The numerical results reported in Tables~\ref{tab:xqx_20},~\ref{tab:xqx_80},~\ref{tab:exp_20} and~\ref{tab:exp_80} show that P-APG solution is very close to the optimal solution of the regularized problem in~\eqref{regularize}. Note that MOSEK using interior point optimizer starts working slowly beyond $N=1600$ due to $\cO(N^2n)$ memory requirement -- see Table~\ref{memory}. Numerical results show that advantages of P-APG over running IPM alone on \eqref{regularize} become more and more evident as the number of observations, $N$, increases.\vspace*{-5mm}
\subsection{Comparison with ASM and ADMM}
\label{sec:tests}
In this section, we compare P-APG with ASM and multi-block ADMM. It is worth noting that multi-block ADMM solves the original problem~\eqref{original}, while P-APG and the active set method~(ASM) solve the regularized problem~\eqref{regularize}. All algorithms are terminated whenever they compute an iterate, $(\py, \pxi)$, satisfying the following stopping criteria:
$\norm{\py - \py^*}/ \sqrt{N} \leq$ 5e-3 and $\norm{ (A_1 \py + A_2 \pxi)_{-}}/\sqrt{N^2-N} \leq$ 1e-1, where the first one is the relative sub-optimality with respect to the original problem in~\eqref{original} and the second one is the normalized infeasibility. The initial point for ASM is set by using~\eqref{eq:slater-point-AS}, where $\alpha = 1/N$. This choice of $\alpha$ works consistently well based on our test.

\begin{figure}[htbp]
\centering
\includegraphics[width=0.48\textwidth]{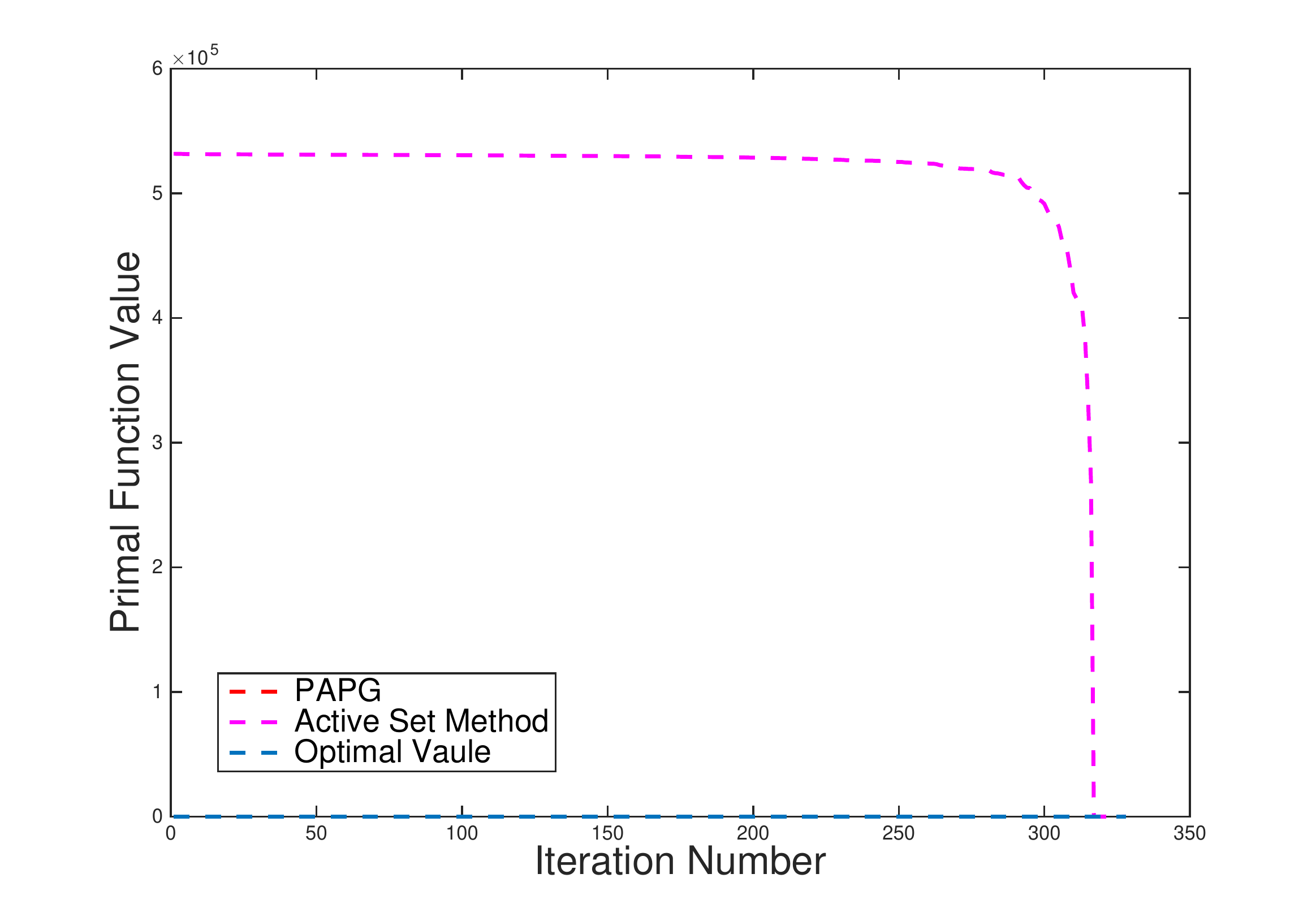}
\includegraphics[width=0.48\textwidth]{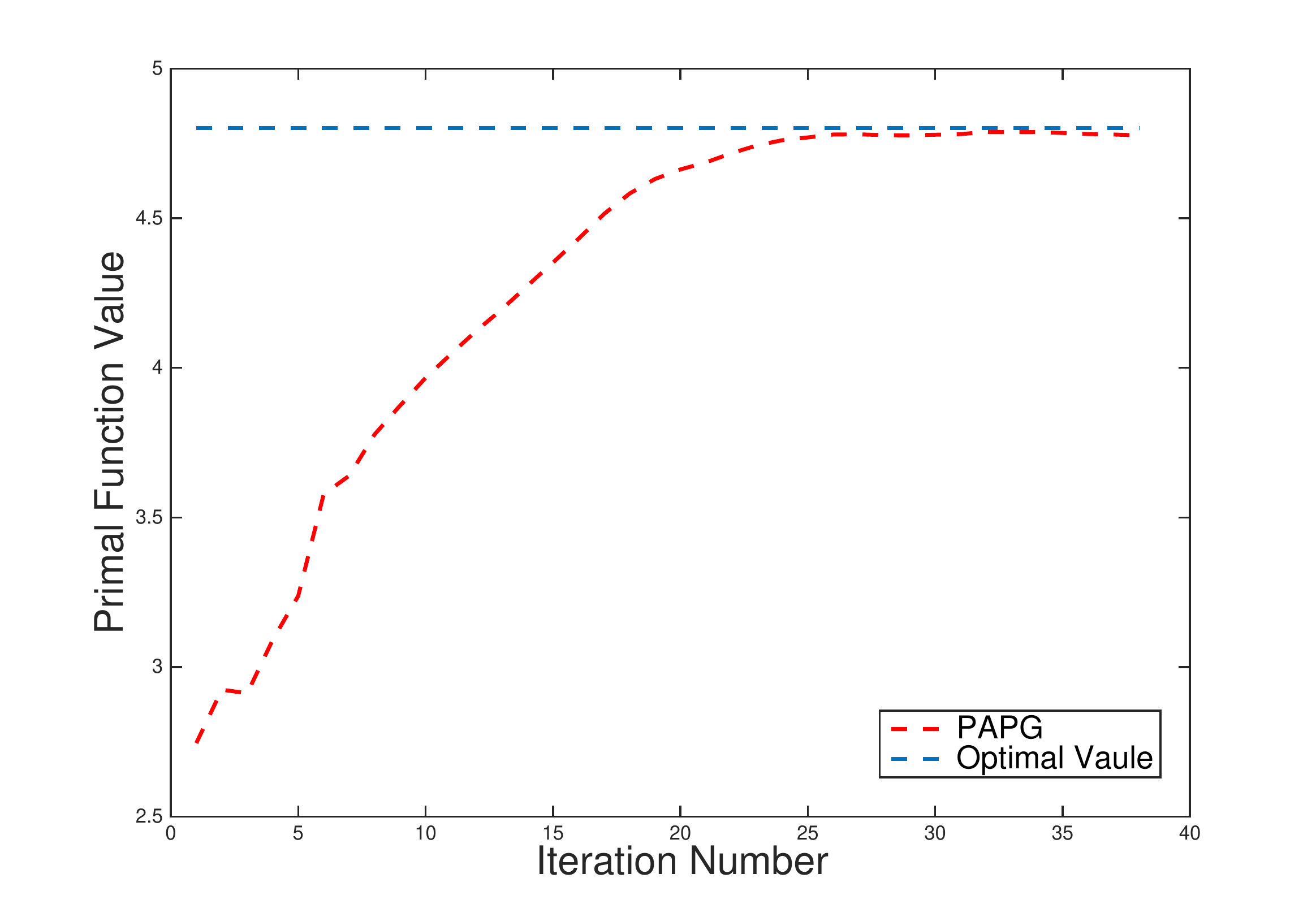}
\caption{P-APG vs ASM $(n,N)=(10,100)$: (left) Relative Suboptimality, (right) Zoom-in for P-APG method}
\label{100_1}
\vspace*{-5mm}
\end{figure}
\begin{figure}[t!]
\centering
\includegraphics[width=0.48\textwidth]{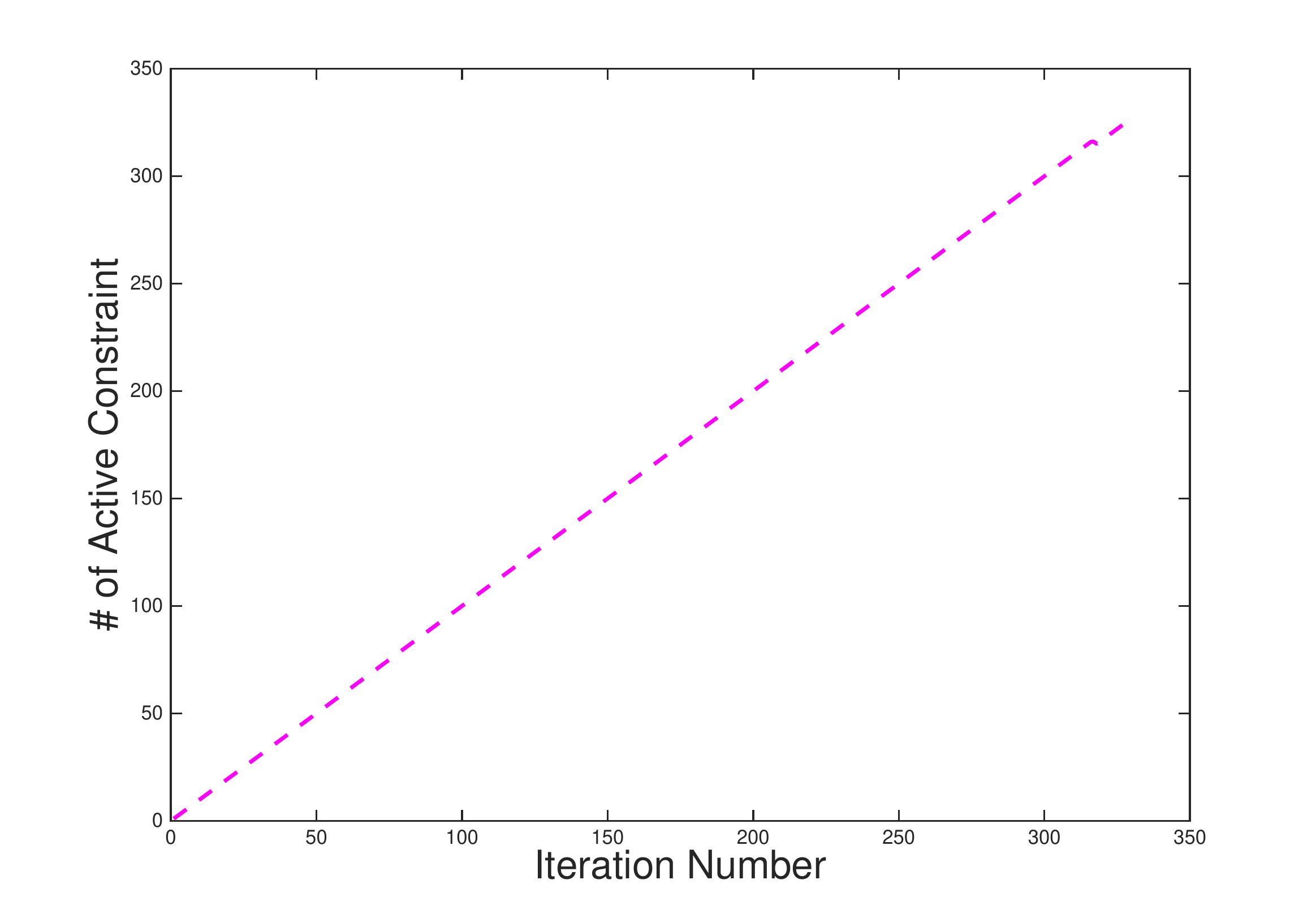}
\includegraphics[width=0.48\textwidth]{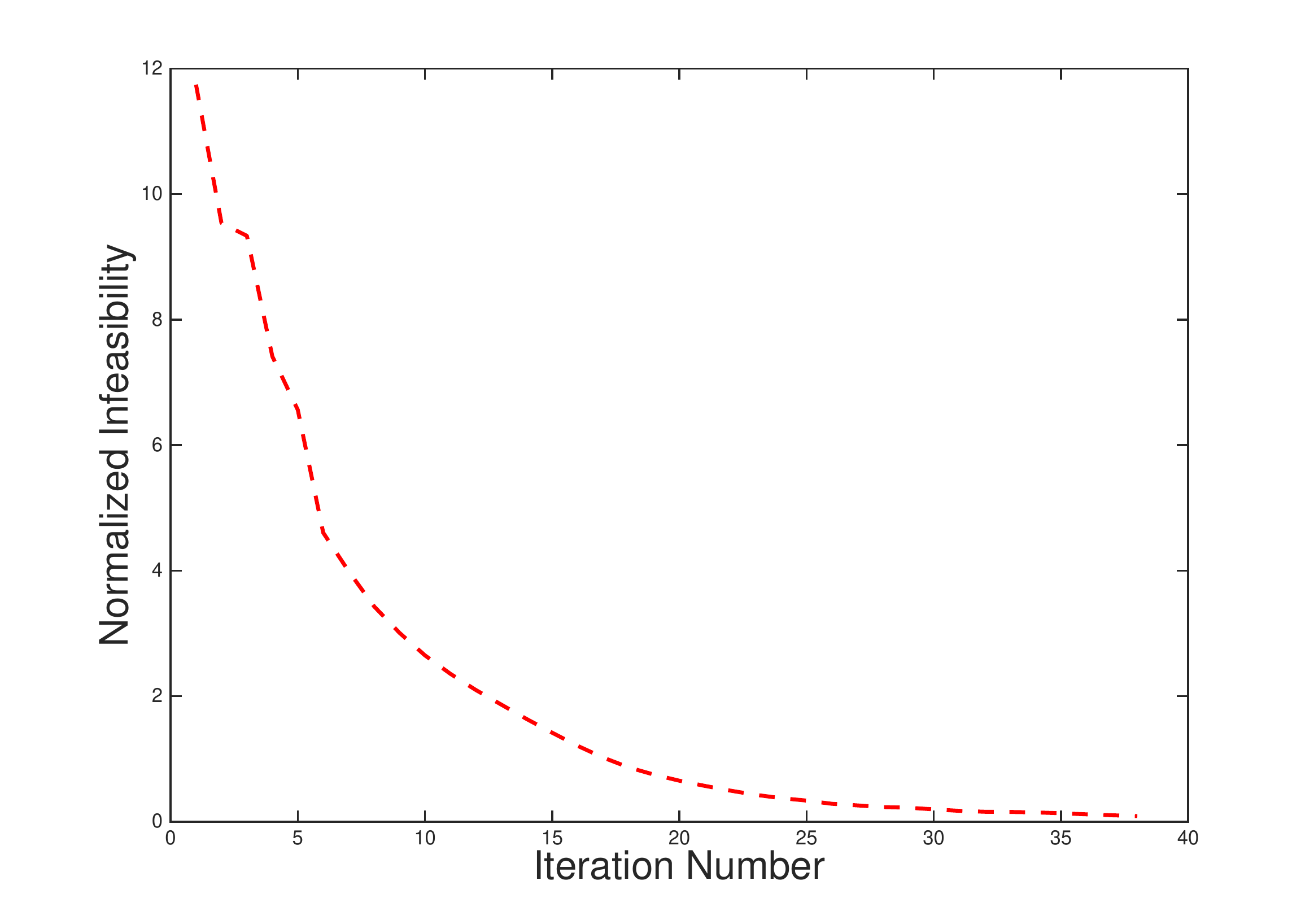}
\caption{P-APG vs ASM $(n,N)=(10,100)$: (left) Number of Active Constraints for ASM, (right) Normalized infeasibility for P-APG}
\label{fig:infeas-PAPG vs ASM}
\end{figure}

Experiments comparing convergence behaviors of P-APG and ASM on a small-size problem were carried out for $(n,N)=(10,100)$. 
As shown in Fig.~\ref{100_1}, 
active set algorithm spends quite long in a \textit{warm-up} phase before making noticeable progress in terms of function value; and this behavior becomes more and more apparent as the size of the problem increases. Fig.~\ref{fig:infeas-PAPG vs ASM}(left) displays how the number of active constraints for ASM changes. 
Fig.~\ref{fig:infeas-PAPG vs ASM}(right) shows the distance to the feasible region for P-APG method, which converges to zero very fast regardless of the dimension of the problems in all of our tests. In summary, the issues with the active set method are: (i) the majority of the time is spent for identifying the optimal active set before making a noticeable progress in terms of suboptimality; (ii) as the number of active constraint in the algorithm increases, solving the KKT system in~\eqref{eq:KKT-ASM} becomes costly -- this operation is similar to the factorization steps in interior-point methods.

In Tables~\ref{papg_asm_admm_xqx} and~\ref{papg_asm_admm_exp}, besides the statistics reported in Section~\ref{sec:self-test}, we 
also report \texttt{Accuracy} which measures the solution quality with respect to the original problem in \eqref{original}. In particular, given an approximate solution $\tilde{\py}$, obtained by solving either \eqref{original} or \eqref{regularize} depending on the algorithm chosen, \texttt{Accuracy} is computed as $\norm{\tilde{\py} - \py^*}/ \sqrt{N}$.
As in Section~\ref{sec:self-test}, \texttt{Preprocess} for P-APG method denotes the wall-clock time used for computing the maximum singular value for matrix $A_4$, and \texttt{Preprocess} for ADMM accounts for $\bar{\Delta}_j$ computation for all $j\in\cN$ as shown in Figure~\ref{fig:admm}.  The performance comparison is shown in Table~\ref{papg_asm_admm_xqx} and Table~\ref{papg_asm_admm_exp}, which clearly display that as the number of observations $N$ increases, ASM starts struggling to finish the job within 2 hours beyond N=800, and the gap between P-APG and ADMM closes rapidly, and eventually P-APG outperforms ADMM at $N = 2400$, which is also demonstrated in Fig.~\ref{fig:walltime_ratio}.
\begin{figure}[h!]
\begin{center}
\includegraphics[width=0.75\textwidth]{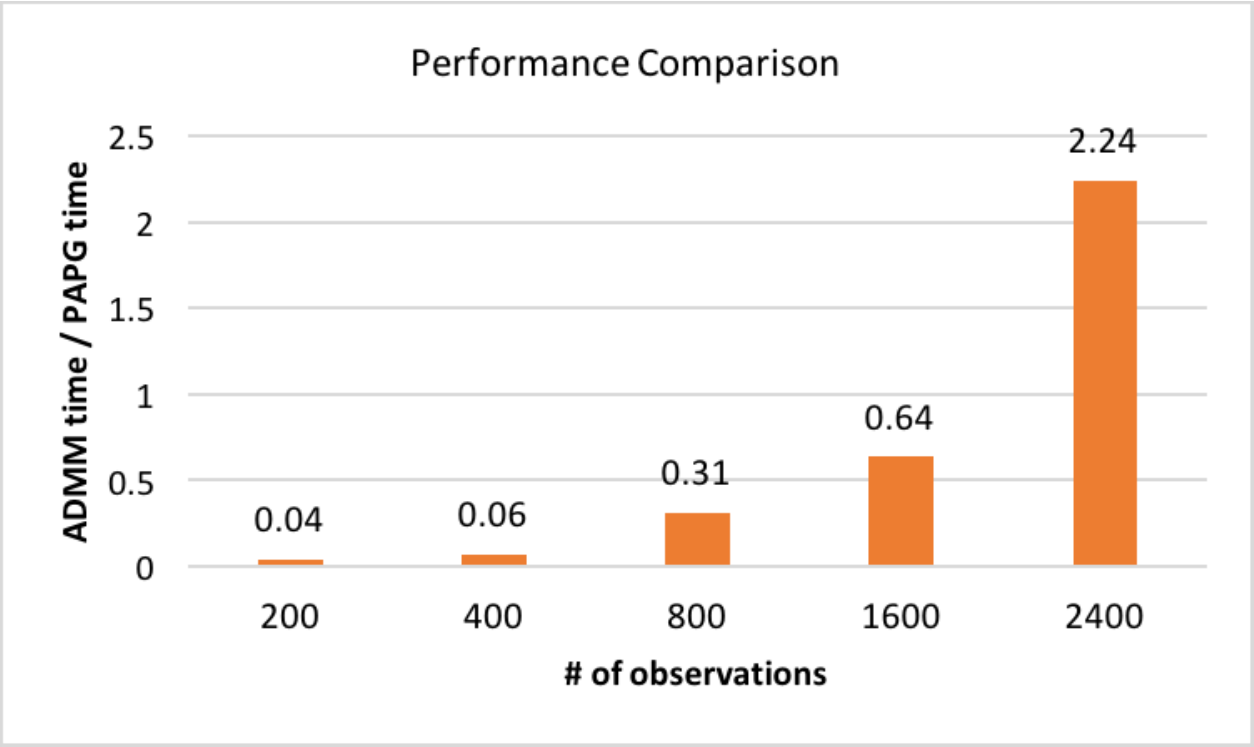}
\caption{Wall-time Ratio between ADMM and P-APG}
\label{fig:walltime_ratio}
\end{center}
\end{figure}
\section{Conclusion}
In this paper, we proposed P-APG method to efficiently compute the least squares estimator for large-scale convex regression problems. By relaxing constraints partially, we obtained the separability on the corresponding Lagrangian dual problem. Using Tikhonov regularization, we ensured the feasibility of iterates in the limit, and we provided error bounds on 1) the distance between the inexact solution to the regularized problem and the optimal solution to the original problem, 2) the constraint violation of the regularized solution. We also proposed a continuation scheme which directly solves the (unregularized) original problem (without any negative impact on the iteration complexity), and it does not require a parameter input depending on the desired solution tolerance $\epsilon$. The comparison in the numerical section demonstrates the efficiency of P-APG method on memory usage compared to IPM. Furthermore, our numerical tests show that P-APG becomes the method of choice for large $N$ values when compared to ASM and ADMM. 
\begin{table}[b!]
\caption{Comparison of PAPG and other methods $\frac{1}{2}\pmb{x}^{\mathsf{T}}Q\pmb{x}$
}
\centering
\setlength{\tabcolsep}{2pt}
{\scriptsize
\begin{tabular}{clccclllll}
\toprule
\texttt{n, N} & \texttt{Algorithms} 
& \texttt{Cores/RAM} & \texttt{Preprocess} & \texttt{Wall-time\hspace{2mm}} & \texttt{Infeasibility} & \texttt{Accuracy\hspace{2mm}} & \texttt{SubOpt\_Reg} \\
\midrule
\multirow{6}{*}{80, 200}
 & Mosek 
 & 2/2 & 0 & 5 & 0 & 0 & -- \\
 & Mosek\_Reg 
 & 2/2 & 0 & 9 & 0 & 1.30E-03 & 0 \\
 & ADMM 
 & 2/2 & 3 & 4 & 5.48E-04 & 4.32E-03 & -- \\
 & ASM 
 & 2/2 & 0 & 111 & 0 & 1.31E-03 & 9.33E-06 \\
 & PAPG\_A 
 & 2/2 & 1 & 105 & 7.29E-02 & 1.31E-03 & 6.54E-05 \\
 & PAPG\_C 
 & 2/2 & 0 & 43 & 8.43E-02 & 1.31E-03 & 3.23E-06 \\
\midrule
\multirow{6}{*}{80, 400}
 & Mosek 
 & 4/4 & 0 & 20 & 0 & 0 & -- \\
 & Mosek\_Reg 
 & 4/4 & 0 & 36 & 0 & 2.60E-03 & 0 \\
 & ADMM 
 & 4/4 & 11 & 20 & 4.46E-04 & 4.71E-03 & -- \\
 & ASM 
 & 4/4 & 0 & 665 & 0 & 2.57E-03 & 5.10E-07 \\
 & PAPG\_A 
 & 4/4 & 3 & 319 & 9.38E-02 & 2.56E-03 & 1.17E-03 \\
 & PAPG\_C 
 & 4/4 & 3 & 176 & 9.79E-02 & 2.56E-03 & 5.72E-04 \\
\midrule
\multirow{6}{*}{80, 800}
 & Mosek 
 & 8/8 & 0 & 109 & 0 & 0 & -- \\
 & Mosek\_Reg 
 & 8/8 & 0 & 188 & 0 & 3.80E-03 & 0 \\
 & ADMM 
 & 8/8 & 50 & 121 & 3.78E-04 & 4.98E-03 & -- \\
 & ASM 
 & 8/8 & 0 & 7006 & 0 & 3.82E-03 & 1.47E-05 \\
 & PAPG\_A 
 & 8/8 & 12 & 391 & 7.55E-02 & 3.81E-03 & 3.73E-04 \\
 & PAPG\_C 
 & 8/8 & 12 & 281 & 9.89E-02 & 3.81E-03 & 2.17E-04 \\
\midrule
\multirow{6}{*}{80, 1600}
 & Mosek 
 & 16/16 & 0 & 544 & 0 & 0 & -- \\
 & Mosek\_Reg 
 & 16/16 & 0 & 918 & 0 & 2.40E-03 & 0 \\
 & ADMM 
 & 16/16 & 221 & 537 & 2.93E-04 & 4.96E-03 & -- \\
 & ASM 
 & 16/16 & 0 & $>$2 \texttt{hours} & \texttt{N/A} & \texttt{N/A} & \texttt{N/A} \\
 & PAPG\_A 
 & 16/16 & 50 & 844 & 9.68E-02 & 3.16E-03 & 2.22E-03 \\
 & PAPG\_C 
 & 16/16 & 50 & 802 & 9.96E-02 & 3.17E-03 & 2.06E-03 \\
\midrule
\multirow{6}{*}{80, 2400}
 & Mosek 
 & 24/24 & 0 & 2537 & 0 & 0 & -- \\
 & Mosek\_Reg 
 & 24/24 & 0 & 4576 & 0 & 3.40E-03 & 0 \\
 & ADMM 
 & 24/24 & 678 & 2332 & 2.16E-04 & 4.99E-03 & -- \\
 & ASM 
 & 24/24 & 0 & $>$2 \texttt{hours} & \texttt{N/A} & \texttt{N/A} & \texttt{N/A} \\
 & PAPG\_A 
 & 24/24 & 155 & 1040 & 9.69E-02 & 3.26E-03 & 9.24E-04 \\
 & PAPG\_C 
 & 24/24 & 155 & 1184 & 9.88E-02 & 3.26E-03 & 1.30E-03 \\
\bottomrule
\end{tabular}}
\label{papg_asm_admm_xqx}
\end{table}%
\clearpage
\begin{table}[t!]
\caption{Comparison of PAPG and other methods $\exp( \pmb{p}^{\mathsf{T}}\pmb{x}) $}
\centering
\setlength{\tabcolsep}{2pt}
{\scriptsize
\begin{tabular}{clccclllll}
\toprule
\texttt{n, N} &\texttt{Algorithms} 
& \texttt{Cores/RAM} & \texttt{Preprocess} & \texttt{Wall-time\hspace{2mm}} & \texttt{Infeasibility} & \texttt{Accuracy\hspace{2mm}} & \texttt{SubOpt\_Reg} \\
\midrule
\multirow{6}{*}{80, 200}
 & Mosek 
 & 2/2 & 0 & 5 & 0 & 0 & -- \\
 & Mosek\_Reg 
 & 2/2 & 0 & 7 & 0 & 4.41E-04 & 0 \\
 & ADMM 
 & 2/2 & 3 & 4 & 6.41E-04 & 4.96E-03 & -- \\
 & ASM 
 & 2/2 & 0 & 155 & 0 & 4.41E-04 & 2.51E-05 \\
 & PAPG\_A 
 & 2/2 & 0 & 64 & 9.21E-02 & 4.43E-04 & 6.68E-03 \\
 & PAPG\_C 
 & 2/2 & 0 & 41 & 9.27E-02 & 4.42E-04 & 1.60E-03 \\
\midrule
\multirow{6}{*}{80, 400}
 & Mosek 
 & 4/4 & 0 & 19 & 0 & 0 & -- \\
 & Mosek\_Reg 
 & 4/4 & 0 & 33 & 0 & 9.28E-04 & 0 \\
 & ADMM 
 & 4/4 & 11 & 18 & 5.51E-04 & 4.95E-03 & -- \\
 & ASM 
 & 4/4 & 0 & 827 & 0 & 9.29E-04 & 2.90E-05 \\
 & PAPG\_A 
 & 4/4 & 3 & 167 & 9.39E-02 & 9.24E-04 & 3.00E-03 \\
 & PAPG\_C 
 & 4/4 & 3 & 143 & 9.63E-02 & 9.24E-04 & 2.92E-03 \\
\midrule
\multirow{6}{*}{80, 800}
 & Mosek 
 & 8/8 & 0 & 136 & 0 & 0 & -- \\
 & Mosek\_Reg 
 & 8/8 & 0 & 175 & 0 & 9.74E-04 & 0 \\
 & ADMM 
 & 8/8 & 50 & 97 & 3.50E-04 & 4.98E-03 & -- \\
 & ASM 
 & 8/8 & 0 & $>$ 2 \texttt{hours} & \texttt{N/A} & \texttt{N/A} & \texttt{N/A} \\
 & PAPG\_A 
 & 8/8 & 13 & 236 & 9.12E-02 & 9.72E-04 & 9.44E-04 \\
 & PAPG\_C 
 & 8/8 & 13 & 209 & 9.87E-02 & 9.69E-04 & 1.65E-04 \\
\midrule
\multirow{6}{*}{80, 1600}
 & Mosek 
 & 16/16 & 0 & 843 & 0 & 0 & -- \\
 & Mosek\_Reg 
 & 16/16 & 0 & 1107 & 0 & 1.40E-03 & 0 \\
 & ADMM 
 & 16/16 & 266 & 787 & 2.66E-04 & 4.98E-03 & -- \\
 & ASM 
 & 16/16 & 0 & $>$ 2 \texttt{hours} & \texttt{N/A} & \texttt{N/A} & \texttt{N/A} \\
 & PAPG\_A 
 & 16/16 & 62 & 461 & 9.72E-02 & 1.45E-03 & 3.07E-03 \\
 & PAPG\_C 
 & 16/16 & 59 & 742 & 9.95E-02 & 1.44E-03 & 1.28E-04 \\
\midrule
\multirow{6}{*}{80, 2400}
 & Mosek 
 & 24/24 & 0 & 2522 & 0 & 0 & -- \\
 & Mosek\_Reg 
 & 24/24 & 0 & 3365 & 0 & 1.80E-03 & 0 \\
 & ADMM 
 & 24/24 & 846 & 2486 & 2.12E-04 & 4.95E-03 & -- \\
 & ASM 
 & 24/24 & 0 & $>$ 2 \texttt{hours} & \texttt{N/A} & \texttt{N/A} & \texttt{N/A} \\
 & PAPG\_A 
 & 24/24 & 138 & 774 & 9.69E-02 & 1.82E-03 & 7.82E-03 \\
 & PAPG\_C 
 & 24/24 & 139 & 1379 & 9.94E-02 & 1.81E-03 & 2.23E-03 \\
\bottomrule
\end{tabular}
}
\label{papg_asm_admm_exp}
\end{table}%
\bibliographystyle{spmpsci}      
\bibliography{paper}   


\end{document}